\documentclass{elsarticle}

\usepackage{graphicx}
\usepackage{amssymb}
\usepackage{amsthm}
\usepackage{lineno}
\usepackage{tikz}
\usepackage{hyperref}
\usepackage[utf8]{inputenc}
\usepackage{fixltx2e,fix-cm}
\usepackage{amsmath}
\usepackage{amsfonts}
\usepackage{subfigure}
\usepackage{color}
\usepackage{enumerate}
\usepackage{makeidx}
\usepackage{multicol}
\usepackage{lipsum}
\usepackage{comment}
\usetikzlibrary{decorations.markings}
\usepackage{mathtools}
\usepackage{tabularx}
\usepackage{footnote}
\makesavenoteenv{tabularx}
\usepackage{pdfpages}

\makeatletter
\def\ps@pprintTitle{%
 \let\@oddhead\@empty
 \let\@evenhead\@empty
 \def\@oddfoot{}%
 \let\@evenfoot\@oddfoot}
\makeatother

\tikzset{Node Label Style/.style={style={draw,circle}, fill=white, minimum size=0.5cm}}

\newtheorem{theorem}{Theorem}[section]
\newtheorem{proposition}[theorem]{Proposition}
\newtheorem{lemma}[theorem]{Lemma}
\newtheorem{corollary}[theorem]{Corollary}

\theoremstyle{definition}
\newtheorem{definition}[theorem]{Definition}

\theoremstyle{remark}
\newtheorem{remark}[theorem]{Remark}
\newtheorem{example}[theorem]{Example}

\theoremstyle{plain}

\theoremstyle{plain}

\newenvironment{decisionproblem}[1]%
	{\vskip\topsep\noindent\textsc{#1.\/}}%
	{\par\vskip\topsep}%

\newenvironment{amatrix}[1]{%
  \left[\begin{array}{@{}*{#1}{c}|c@{}}
}{%
  \end{array}\right]
}

\graphicspath{{figures/}}

\begin{document}

\begin{frontmatter}

\title{Computational complexity and pragmatic solutions for flexible tile based DNA self-assembly}

\author[label1]{Leyda Almod\'ovar}
\address[label1]{Stonehill College, Easton, MA}
\ead{lalmodovarvel@stonehill.edu}

\author[label2]{Jo Ellis-Monaghan}
\address[label2]{University of Amsterdam, the Netherlands}
\ead{jellismonaghan@gmail.com}

\author[label3]{Amanda Harsy}
\address[label3]{Lewis University, Romeoville, IL}
\ead{harsyram@lewisu.edu}

\author[label4]{Cory Johnson}
\address[label4]{California State University, San Bernardino, San Bernardino, CA}
\ead{corrine.johnson@csusb.edu}

\author[label5]{Jessica Sorrells}
\address[label5]{Converse University, Spartanburg, SC}
\ead{jessica.sorrells@converse.edu}

\begin{abstract}
 Branched junction molecule assembly of DNA nanostructures, pioneered by Seeman's laboratory in the 1980s, has become increasingly sophisticated, as have the assembly targets. A critical design step is finding minimal sets of branched junction molecules that will self-assemble into target structures without unwanted substructures forming.  We use graph theory, which is a  natural design tool for self-assembling DNA complexes, to address 
 this problem. After determining that finding optimal design strategies for this method is generally NP-complete, we provide pragmatic solutions in the form of programs for special settings and provably optimal solutions for natural assembly targets such as platonic solids, regular lattices, and nanotubes. These examples also illustrate the range of design challenges.  
\end{abstract}


\end{frontmatter}

\newpage

%\linenumbers

\section{Introduction \& Background}  

Branched junction molecule self-assembly of DNA nanostructures, pioneered by Seeman's laboratory in the 1980s, has become increasingly sophisticated, as have the assembly targets. 
In this method, multi-armed DNA molecules with controlled sequences of unsatisfied sites on the ends of their arms bond with one another to self-assemble into a desired shape.  \emph{Tiles} are combinatorial abstractions of these branched junction molecules. We address the fundamental design problem of determining  optimal sets of tiles that will self-assemble into the targeted shapes.

We use graph theory, which is a  natural design tool for self-assembling DNA complexes, to model  self-assembly from flexible-armed tiles. Results presented here add to the current knowledge of optimal solutions and algorithms for different classes of graphs, including lattice graphs, which can be used in nanotube construction. This includes creating general design theory and finding accurate bounds for the number of tile types and bond-edge types for a variety of new graphs. Questions regarding design strategies for realizing a target graph are considered under three different scenarios of graded levels of restriction. 

We begin by establishing the computational complexity of two common design challenges, namely determining what a given collection of tiles will produce, and determining whether a collection of tiles that will produce a desired target structure will also produce unwanted incidental smaller structures.  We prove that in both cases the problem is intractible. This means that unless P=NP, fast general algorithms are unfeasible.  Thus, we turn to explicit designs for high-utility graphs such as the platonic solids and lattices.  While provably optimal designs are possible for these special graphs, the designs in turn reveal new challenges.  The examples given here highlight some of those challenges.  

Some of the target graphs we analyze are mathematically standard, such as the platonic solids. The hexahedron (cube graph) is particularly notable in terms of DNA nanostructures. The DNA cube was first successfully constructed from branched-junction molecules in 1991 \cite{chen1991synthesis}, signifying the first laboratory achievement of a ``closed'' DNA nanostructure. The ability to produce a cube established DNA as a material that could form functional structures via self-assembly. In general, the platonic solids are $k$-regular, thus optimal pots for the least restrictive scenario were determined previously in \cite{ellis2014minimal}. Optimal pots for more restrictive scenarios were previously known only for the tetrahedron. We have determined optimal solutions in more restrictive scenarios for the hexahedron, octahedron, and icosahedron, as well as upper bounds for the dodecahedron (see \cite{repository}).

Other graphs we have chosen because of their applications. DNA nanotubes are a fundamental form used for molecular channeling, drug delivery, and biomolecular sensing \cite{rothemund2004design, wilner2011self}. We investigate both planar lattices and tube structures formed by identifying edges of lattices. ``Closing'' tile lattices to form tubes by identifying lattice boundaries is a method frequently used in laboratories, but finding efficient ways to construct these tubes has proven challenging \cite{Liu2019, wilner2011self}.  In this work we provide optimal solutions for all sizes of square and triangular lattice tubes under the conditions of the least restrictive scenario. In addition to the proofs directly included here, we provide a repository of results for a variety of graph families under different restrictive conditions (again, see \cite{repository}). 

Furthermore, the examples here demonstrate the difficulty of finding solutions in the more restrictive scenarios.  A system of linear equations can give lower bounds, but due to the complexity results mentioned above, only in special settings. We illustrate the possibility of two non-isomorphic graphs forming from the unique solution of the linear system, again highlighting the difficulty of this problem in more restrictive scenarios. The triangle lattice graph example chosen gives optimal solutions in the most restrictive scenario for both the number of molecule structure types and the number of DNA strand types, but proves that the two cannot be realized simultaneously. 

Our work provides rigorous theoretical tools for the emergent science of DNA-self assembly.  Addressing the problem of minimizing the number of molecule structures and cohesive-end types is essential especially in the case where a wet-lab design fails since having fewer tiles and bonds can make it easier to determine the error for the failed design \cite{hansen2018DNA}.  Moreover, the parameters investigated here, such as the minimum number of different tiles needed to realize a particular graph $G$, are new graph invariants and thus of intrinsic combinatorial interest.  In determining these parameters for various families of graphs, we lay the foundations for future work discovering what structural information they might encode.

\subsection{Paper Organization}

The remainder of the paper is structured as follows. In Section \ref{sec:background} we give an overview of the history and applications of DNA self-assembly. In Section \ref{sec:definitions} we provide some basic definitions and terminology used throughout this paper and also describe the various laboratory conditions considered in our work. Section \ref{sec:prob diff} describes the computational  complexity  of determining  what  graphs  may  be  constructed  from  a  given  collection of molecule types and of assuring that no undesired substructures will form. 
In Section \ref{sec:code}, we discuss software we created that assists with the optimization design problems this research tackles. We share selected results from platonic solids, lattice, and tube graphs in Section \ref{sec:SelectedResults}. The graph results we share were chosen due to their unique results. Other results are outlined in Figure \ref{tableofresults} with details shared in \cite{repository}. 

\subsection{DNA Self-Assembly Background}\label{sec:background}

DNA self-assembly, and self-assembly in general, is a rapidly advancing field, with \cite{72, 78} providing good overviews. Synthetic DNA molecules have been designed to self-assemble into given nanostructures, starting with branched DNA molecules \cite{59, 89}, nanoscale arrays \cite{91, 92}, numerous polyhedra \cite{benson2015DNA, chen1991synthesis, he2008hierarchical,  49, 84, zhang1994construction}, arbitrary graphs \cite{55, 75, 94}, a variety of DNA and RNA knots \cite{64, 65, 88}, and the first macroscopic self-assembled 3D DNA crystals \cite{99}. This has led to molecular scaffoldings made of DNA and other proteins which have wide-ranging potential \cite{steph2020}. Such applications include construction of containers for the transport and release of nano-cargos, templates for the controlled growth of nano-objects, biomolecular computing, biosensors, fine screen filters (lattices) at the nano-size scale, nano-circuitry, robotics, and drug-delivery methods (see \cite{adleman1994molecular, ellis2019tile, ferrari2018, 41, 43, 61, labean2007constructing, 67, 69, seeman2007overview, wickham2012DNA,  yan2003DNA, 70}). In particular, DNA nanotechnology is arguably the most powerful method for building nanoscale structures resembling macroscopic robots \cite{gerling2015}. In 2012, polyhedral cages were self-assembled from branched DNA tiles using biotin to guide the process \cite{zhang2012}. DNA ``origami" methods have been utilized in biophysics for nucleosome assembly \cite{funke20162, funke2016, le2016}.
  
DNA-based designs have increased rapidly in recent years because they are relatively cheap and easy to synthesize \cite{hansen2018DNA}. Furthermore, their predictable behavior and structure allows them to be programmable. Most pharmaceutical
companies have  started utilizing this DNA-based technology to design new drug-like molecules \cite{goodnow2017dna}. Since 1994, the use of computer algorithms has become a powerful tool to help solve the combinatorial optimization problems within DNA-based designs \cite{adleman1994molecular}.   Specifically, mathematical methods especially in the realm of  linear programming and graph theory have be used to help optimize automate design strategies for DNA-templated programs \cite{ andersen2015towards, hansen2018DNA}.

Graph theory can expedite the process of designing component molecular building blocks for self-assembly since the target structures are often wire frames such as polyhedral skeletons, surface meshes, cages, lattice subsets, and other graph-like structures. See \cite{jonoska2006spectrum} for some of the first graph theoretical foundations of DNA self-assembly designs. Self-assembling DNA molecules have already been modeled by graphs which represent cubes \cite{chen1991synthesis}, octahedra \cite{rothemund2006folding, zhang1994construction}, and polyhedra including buckyballs, tetrahedra, and dodecahedra \cite{he2008hierarchical}. In \cite{ellis2013example}, Ellis-Monaghan and Pangborn discuss several other mathematical techniques which may be used to aid in solving these design problems including origami and linear strand threading methods. Ellis-Monaghan et al. and LaBean and Li give a survey of construction methods in \cite{ ellis2019tile, labean2007constructing}.

\subsection{Graph Theoretical Definitions and Notation}\label{sec:definitions}  
We use the graph theoretical formalism of \cite{ellis2019tile}, and provide some of the essential constructions here  for the convenience of the reader (see also \cite{jonoska2006spectrum}).  A discrete graph $G$ consists of a set $V=V(G)$ of vertices and a set $E=E(G)$ of edges together with a map $\mu:E \rightarrow V^{(2)}$  where $V^{(2)}$ is the set of (not necessarily distinct) unordered pairs of elements of $V$. If $\mu(e)=\{u,v\}$, then $u$ and $v$ are the vertices incident with $e$. We allow graphs to have loops and multiple edges, so that it is possible for $\mu(e) = (u,v)$ with $u=v$ (in the case of a loop edge) or for $\mu(e)=\mu(f)$ (in the case of multiple edges). We denote a \emph{half-edge} of a vertex $v$ as $(v,e)$ if $v \in \mu(e)$. Throughout we use the notation $\#$ to denote size or quantity, so that e.g. $\#V(G)$ will denote the order of $G$.

\begin{definition} A \emph{k-armed branched junction molecule} is a star-shaped molecule whose arms are formed from strands of DNA, possibly multiple strands.  At the end of each of these arms is a region of unsatisfied bases, forming a \emph{cohesive-end} (also known informally as a \emph{sticky end}).  Arms with complementary cohesive-ends can bond via Watson-Crick base pairing. \end{definition}

In the simplest setting, the arms of a branched junction molecule are double stranded DNA with one strand extending beyond the other to form a cohesive region (see \cite{SK94}).  In this case, the molecules are quite flexible.  Thus, the mathematical model here assumes no geometric restrictions, for example on the inter-arm angles, edge lengths, or ability for any one arm to connect with any other.  See \cite{ferrari2018} for a model that encompasses such geometric constraints.

  \begin{figure} \centering \includegraphics[width=2.5in]{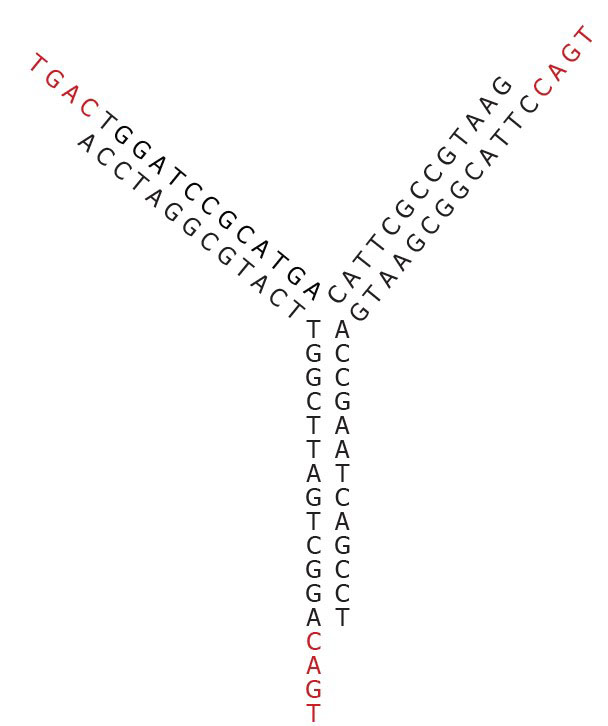} \caption{Representation of a branched junction molecule} \end{figure}

\begin{definition}
A \emph{cohesive-end type} is an element of a finite set $S = \Sigma \cup \hat{\Sigma} $ of hatted and un-hatted symbols.  Each cohesive-end type corresponds to a distinct arrangement of bases forming a cohesive-end on the end of a branched junction molecule arm, such that a hatted and an un-hatted symbol, say $a$ and $\hat{a}$, correspond to complementary cohesive-ends.  We do not allow palindromic cohesive-ends,  so a cohesive-end corresponding to 
 $a$ is complementary to but different from the cohesive-end corresponding to $\hat{a}$ for all $a\in \Sigma$. Moreover we use the convention that $\hat{\hat a} =a$. 

\end{definition}

\begin{definition}
A cohesive-end type joined to its complement forms a \emph{bond-edge type}, which we identify by the un-hatted symbol, so for example, cohesive-end types $a$ and $\hat{a}$ will join to form a bond-edge of type $a$.
\end{definition}

\begin{definition}  A $k$-armed branched junction molecule is represented in the mathematical model by a vertex of degree $k$ with $k$ incident half-edges called a \emph{tile}. \footnote{Such tiles are sometimes called `flexible tiles' to distinguish them from the small, rigid tiles of Wang tilings or tile assembly models (TAM), surveyed for example in \cite{CW2017}.  However, for simplicity we use just the word `tile' here.}

The half-edges of a tile are labeled by the cohesive-end types corresponding to the cohesive-ends on the arms of the molecule the tile represents. A tile is denoted by a multi-set of its cohesive-end types whose multiple entries of the same cohesive-end type are indicated by the exponent to the corresponding symbol. See Figure \ref{potoftiles}. The number of arms of a tile $t$ is denoted by $\#t$.  
\end{definition}

\begin{figure}\centering \includegraphics[scale=.25]{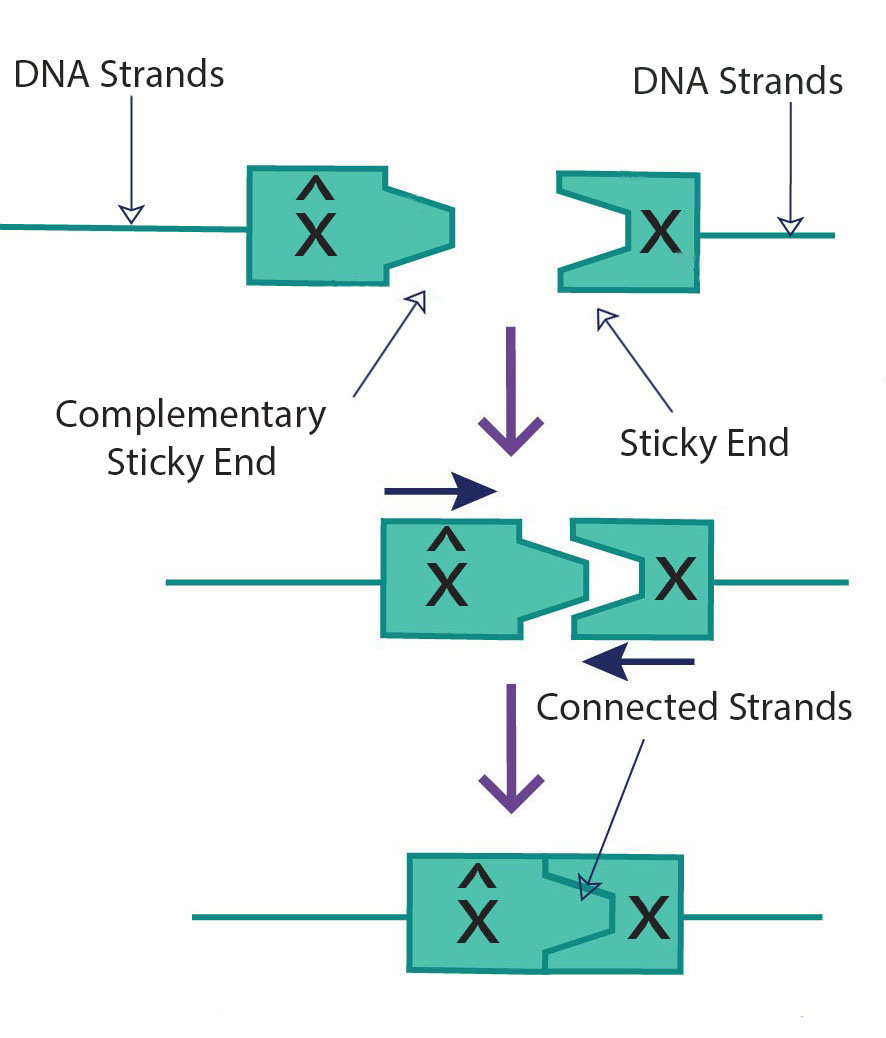}\caption{Formation of a bond-edge} \end{figure}
 
\begin{figure} \centering \includegraphics[scale=.7]{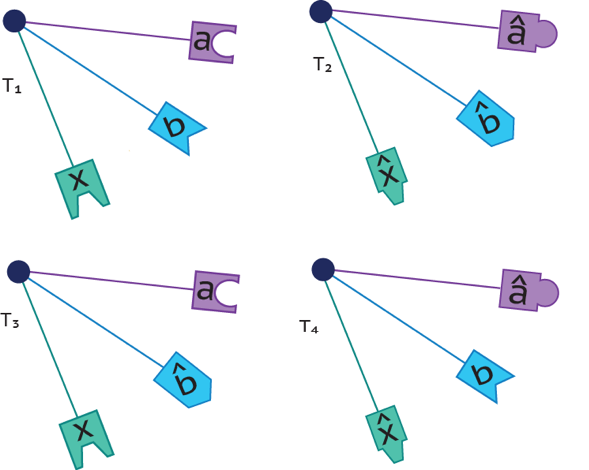}\caption{Pot of tiles, $P = \{ \{a,b,x\}, \{\hat{a}, \hat{b}, \hat{x} \}, \{a, \hat{b}, x \}, \{\hat{a}, b, \hat{x}\}\}$}\label{potoftiles} \end{figure}

\begin{definition} A \emph{pot} is a collection of tiles such that if a cohesive-end type appears in the multi-set of any tile in the pot, its complement also appears on some tile in the pot. More precisely, a \textit{pot} is a set $P = \{t_1,...,t_k\}$ where each $t_i$ is a tile and for all $a\in \Sigma$, if there is $i$ such that $a\in t_i$ then there is $j\in\{1,...,k\}$ such that $\hat{a}\in t_j$. The set of bond-edge types that appear in tiles of $P$ is denoted with $\Sigma(P)$, and we write $\#P$ to denote the number of distinct tile types in $P$ and $\#\Sigma(P)$ to denote the number of distinct bond-edge types that appear in $P$.  These distinct bond-edge types may be thought of as colors, as in \cite{BF2020}. 
\end{definition}

The objective is to assemble a target graph from the collection of tiles, or conversely to determine which graphs can be assembled from a given collection of tiles.  This requires mapping tiles to the vertices of the target graph, and labels from $S$ to the half-edges. The formal notation for this process is described in the following definitions.

\begin{definition} An \textit{assembly design} of a graph $G$ is a labeling $\lambda : H \rightarrow \Sigma \cup \hat{\Sigma}$ of the half-edges of $G$ with the elements of $\Sigma$ and $\hat{\Sigma}$ such that if $e \in E(G)$ and $\mu(e) = \{u,v\}$, then $\widehat{\lambda(v,e)} = \lambda(u,e)$. This means that each edge of $G$ receives both a hatted and an un-hatted version  of the same symbol, one on each of its half-edges. We use the convention that $\lambda$ provides each edge with an orientation that starts from the un-hatted half-edge to the hatted half-edge.  \end{definition}

\begin{definition} An \emph{assembling pot $P_{\lambda}(G)$} for a graph $G$ with assembly design $\lambda$ is the set $P_{\lambda}(G) = \{t_v \mid v \in V(G)\}$ where $t_v =\{ \lambda (v,e) \mid v \in \mu(e), e \in E(G)\}$. This means that for each vertex $v$ of $G$, the assembly design specifies a tile $t_v$ whose multi-set is the set of labels of half-edges incident to $v$. Note there are two complementary labels when $v$ has a loop, and that in general it is possible for $t_u = t_v$ even when $u \neq v$.  If we view a vertex $v$ as its set of half-edges and a tile as a multiset of labels, then the labeling $\lambda$ can be used to map vertices to tiles by $\lambda : V \rightarrow P_{\lambda}(G)$ such that $\lambda(v) = t_v$. \end{definition}

\begin{definition} We say a pot $P$ \textit{realizes} the graph $G$ if there exists an assembly design $\lambda : H \rightarrow \Sigma \cup \hat{\Sigma}$ such that $P_{\lambda}(G) \subseteq P$. \end{definition} Note that a graph may be realized in more than one way by a pot. For example a pot containing tiles $\{a, \hat{a}\}$ and $\{b, \hat{b}\}$ can realize the graph consisting of a single vertex with one loop in two ways.

\begin{definition} The set of graphs with their associated assembly designs realized by a pot $P$, namely $\{(G,\lambda): P_{\lambda}(G) \subseteq P \}$, is called the \textit{output} of $P$ and is denoted by $\mathcal{O}(P)$. \end{definition}

We consider design process questions regarding which types of final structures can be constructed from a given pot of tiles. Inversely, we can also find a pot of tiles that will realize a given target graph. The goal in answering these questions will be to determine the minimum number of branched junction molecules needed for self-assembling DNA nanostructures as well as their combinatorial structures. Given a target graph $G$, we seek pots $P$ to realize $G$ under three different conditions of varying restriction. 
\begin{itemize}
    \item \emph{Scenario 1.} Least Restrictive: $G \in \mathcal{O}(P)$. Note: This allows the possibility that there exists $H \in \mathcal{O}(P)$ such that $\#V(H) < \#V(G)$. 
    \item \emph{Scenario 2.} Moderately Restrictive: $G \in \mathcal{O}(P)$ and for all $H \in \mathcal{O}(P)$, $\#V(H)\geq \#V(G)$. Note: This allows the possibility that there exists $H \in \mathcal{O}(P)$ such that $\#V(H)=\#V(G)$ but $H \not \cong G$. 
    \item \emph{Scenario 3.} Highly Restrictive: $G \in \mathcal{O}(P)$, and for all $H \in \mathcal{O}(P)$, $\#V(H)\geq\#V(G)$, and if $\#V(H)=\#V(G)$ then $H \cong G$.
\end{itemize} 

For convenience when discussing the three scenarios described above, we adopt the following notation.

\begin{definition}
$T_i(G) = \text{min}\{ \#P \mid P \text{ realizes } G \text{ according to Scenario } i\}$.
\end{definition}

\begin{definition}
$B_i(G) = \text{min}\{ \# \Sigma (P) \mid P \text{ realizes } G \text{ according to Scenario } i\}$.
\end{definition}

 Note that the more restrictive scenarios require that a pot $P$ be specified such that no graphs with fewer vertices than $G$ are in $\mathcal{O}(P)$. To determine the graph on fewest vertices formed from a given pot $P$, it is useful to analyze solutions to the equations that the tile types in $P$ must satisfy. We denote the proportion of tile type $t_i$ used in the assembly process as $r_i$, while the number of times a tile $t_i$ appears in the realization of a graph is denoted as $R_i$.

\begin{definition}\label{matrix} Let $P = \{ t_1,...,t_p \}$ be a pot and let $z_{i,j}$ denote the net number of cohesive-ends of type $a_i$ on tile $t_j$, where un-hatted cohesive-ends are counted positively and hatted cohesive-ends are counted negatively. Then the following system of equations must be satisfied by any connected graph in $\mathcal{O}(P)$:
\begin{eqnarray*} z_{1,1}r_1+z_{1,2}r_2+...+z_{1,p}r_p &=& 0 \\
& \vdots & \\
z_{m,1}r_1+z_{m,2}r_2+...+z_{m,p}r_p &=& 0 \\
r_1 + r_2 + ... + r_p &=& 1 \end{eqnarray*}

The \emph{construction matrix} of $P$, denoted $M(P)$, is the corresponding augmented matrix:

\begin{equation}
M(P) = \begin{bmatrix} \begin{array}{*{20}{cccc|c}}
   {{z_{1,1}}} & {{z_{1,2}}} &  \ldots  & {{z_{1,p}}} & 0  \\
    \vdots  &  \vdots  & {\ddots} &  \vdots  & {} \vdots \\
   {{z_{m,1}}} & {{z_{m,2}}} &  \ldots  & {{z_{m,p}}} & 0  \\
1 & 1 &  \ldots  & 1 & 1  \\
 \end{array} \end{bmatrix}. 
\end{equation}
\end{definition}

\begin{definition} The solution space of the construction matrix of a pot $P$ is called the {\it spectrum} of $P$, and is denoted  $\mathcal S(P)$. \end{definition}

The construction matrix is used primarily to determine whether a given pot satisfies the restrictions of Scenario 2 for a desired target graph. When $M(P)$ has a unique solution $\langle r_1,...,r_p \rangle$ it is easy to see that the smallest order of a graph in $\mathcal{O}(P)$ is the least common denominator of the $r_i$'s.  More generally, we recall the following result from ~\cite{ellis2014minimal}.

\begin{proposition}\label{spectrum}
Let $P=\{t_1, \ldots , t_p\}$ be a pot.  Then:

\begin{enumerate}
    \item If a graph $G$ of order $n$ is realized by $P$ using $R_j$ tiles of type $t_j$, then $\frac{1}{n}\langle R_1, \ldots, R_p\rangle$ is a solution of the construction matrix $M(P)$, i.e. is in $\mathcal{S}(P)$.
    \item If $\langle r_1, \ldots, r_p \rangle \in \mathcal{S}(P)$, and $n$ is a positive integer such that $nr_j \in \mathbb{Z}_{\geq 0}$ for all $j$, then there is a graph of order $n$ in $\mathcal{O}(P)$ that is realized by using $nr_j$ tiles of type $t_j$. 
    \item  The smallest order of a graph in $\mathcal{O}(P)$ is $min \{lcm \{b_j|r_j \neq 0 \text{ and } r_j = a_j/b_j\}\}$ where the minimum is taken over all $\langle r_1,\ldots,r_p\rangle \in \mathcal{S}(P)$ such that $r_j \geq 0$ and $a_j/b_j$ is in reduced form for all $j$.
\end{enumerate}
\end{proposition}

\section{Computational Complexity}\label{sec:prob diff}

Essential to design strategies in Scenarios 2 and 3 is determining what graphs a pot can realize and establishing that nothing smaller than the target graph can be realized by the pot.  Thus, we begin by establishing the computational complexity of two problems, that of determining if a given pot will realize any graph of a given order, and that of determining whether a pot that realizes a given graph can also realize any smaller graph. 

We show that in both cases, the problem is intractable.  These outcomes mean that determining the spectrum of a pot is in general intractable, as is determining if a pot that realizes a targeted graph will also realize incidental smaller complete structures of given size.  Although these are disappointing outcomes from a pragmatic point of view, as with most NP-hard problems, they open a new vista of attractive problems through the need to approach the design objectives now through restricted cases and special classes, such as those explored in this paper.
 \subsection{Intractability of determining the spectrum of a pot}
We first observe that if we ask whether a pot $P$ realizes a complete complex of size $k$, we can let the $r_i$'s in Definition~\ref{matrix} be the number of tiles of each type, instead of the proportion of tiles of each type.  This reduces the question to asking if the following matrix has a solution  $\langle r_1 , \ldots , r_n\rangle$ in $\mathbb{Z}^n$.  

\begin{equation}
M_k(P) = \begin{bmatrix} \begin{array}{*{20}{cccc|c}}
   {{z_{1,1}}} & {{z_{1,2}}} &  \ldots  & {{z_{1,p}}} & 0  \\
    \vdots  &  \vdots  & {\ddots} &  \vdots  & {} \vdots \\
   {{z_{m,1}}} & {{z_{m,2}}} &  \ldots  & {{z_{m,p}}} & 0  \\
1 & 1 &  \ldots  & 1 & k  \\
 \end{array} \end{bmatrix}. 
\end{equation}

Since the entries of $M_k(P)$ are also integers, this is now an integer linear programming (ILP) problem. Integer linear programming is NP-hard in general, but this is a very specific form of an ILP, and thus needs an independent proof of hardness.

The proof we give here uses essential ideas from \cite{55} that were developed there for biomolecular computation, but that here are adapted to our  notation and now used to give explicit proofs of computational complexity.

Here is the decision problem.

\begin{decisionproblem}{Pot Realization Problem (PRP)}
 Given a pot $P$ and a positive integer $k$, does $P$ realize a graph $G$ with $k$ vertices?
\end{decisionproblem}

Thus, the \textsc{PRP} asks if, given a pot $P$ and positive integer $k$, there exists a $k$-vertex graph $G$ and an assembly design $\lambda$ such that $P_{\lambda}(G) \subseteq P$.  
A `yes' instance can be certified simply by checking that each element of $P_{\lambda}(G)$ is in $P$.  Thus, the \textsc{PRP} is in NP.

We will prove that the \textsc{PRP} is NP-complete with a polynomial-time reduction to 3-coloring, namely determining if a graph can be properly colored with three colors.  3-coloring is known to be NP-complete even for simple, 4-regular, planar graphs (see~\cite{Dailey1980}). 

Since determining if a graph of order $k$ is realized by the pot $P$ is equivalent to finding an integer solution to $M_k(P)$, showing that \textsc{PRP} is NP-hard is equivalent to showing that the ILP problem of finding a solution for $M_k(P)$ is NP-hard.

\begin{theorem} \label{PRPhard} The \textsc{PRP} is NP-hard.
\end{theorem}

\begin{proof}
Assume we are given a simple, connected, 4-regular, planar graph $G$ with $m$ edges and $n$ vertices, as well as three colors denoted $r$, $b$, and $g$.  

We construct a pot $P$ as follows.  We begin by labeling the vertices of $G$, say with uppercase letters (or some sufficiently large set of symbols), and also labeling the edges of $G$, say with lowercase letters (or some sufficiently large set of symbols disjoint from those labeling the vertices). As usual, we identify each vertex and edge with their labels. See Figure \ref{LabeledGraph}. 
\begin{figure}
  \centering 
  
  \includegraphics[clip, trim=0.5cm 18cm 0.5cm 2cm, width=0.95\textwidth]{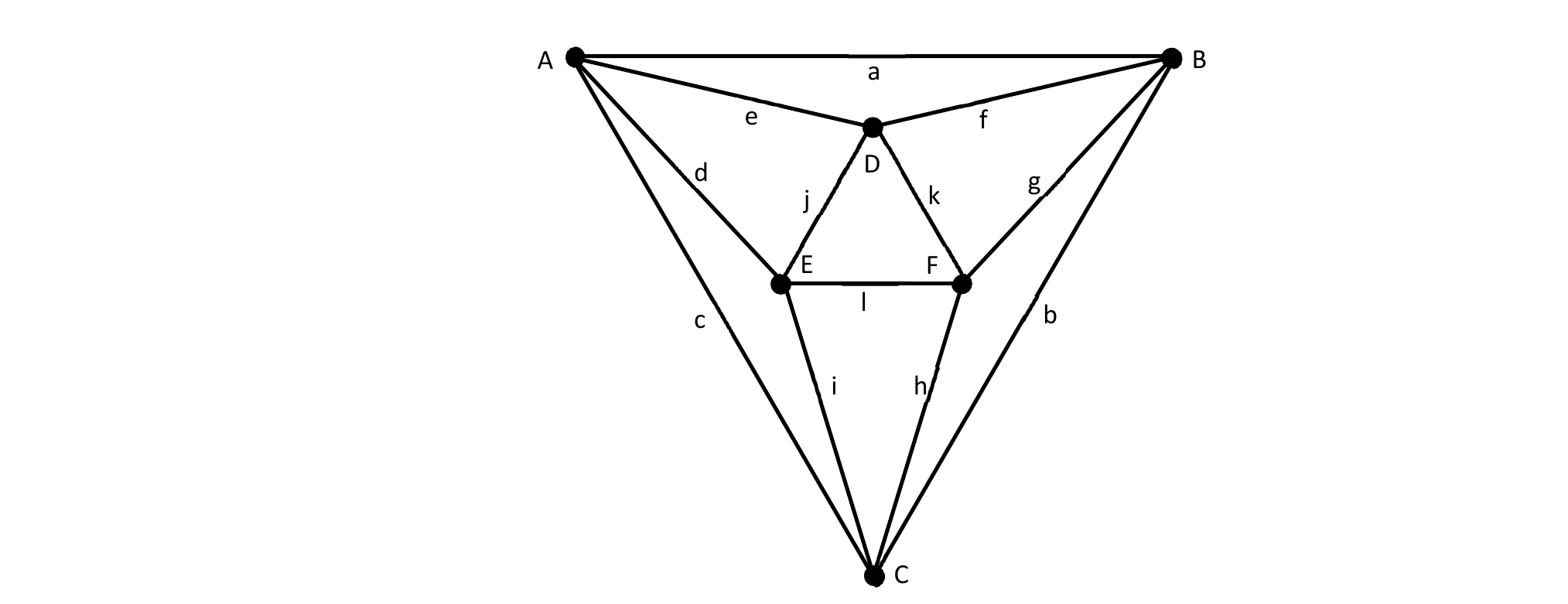}
  %trim=left, bottom, right, top
  \caption{A 4-regular plane graph with labeled vertices and edges.}
  \label{LabeledGraph}
  \end{figure} 
  
For each vertex, we create three 4-armed tiles, one for each color.  Each arm is labeled by the triple of the vertex label, the label of one of the edges incident with the vertex, and a color label, as in Figure~\ref{VertexTilesColored}. 
\begin{figure}
    \centering \includegraphics[width=4.5in]{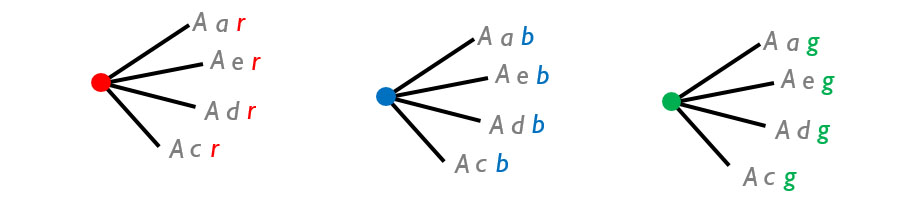} \caption{The three tiles for vertex $A$, one for each color.} 
    \label{VertexTilesColored}
    \end{figure} 
    
This can clearly be done in linear time in the number of vertices.  Then, for each edge, we create six 2-armed tiles.  For each arm, we construct a triple consisting of: the edge label; a vertex label so that the two arms on a tile receive the labels of the two vertices the edge is incident to; and a color label so that the two arms receive different color labels, and so that the six tiles receive all six of the possible ways to color the two ends differently. We then take the `hat' of each triple to label the arms of the 2-armed tiles, as in Figure~\ref{EdgeTilesColored}.   
 \begin{figure}
     \centering \includegraphics[width=3.5in]{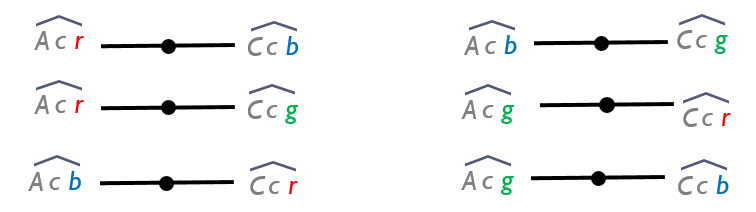} \caption{The six tiles for edge $c$, corresponding to all possible proper coloring of the endpoints of the edge.}
     \label{EdgeTilesColored}
     \end{figure}

Creating these tiles can be done in linear time in the number of edges.

Note that any 2-armed tile can only join two 4-armed tiles with different colors, as in Figure~\ref{ColoredVE}.
\begin{figure}
      \centering \includegraphics[width=4in]{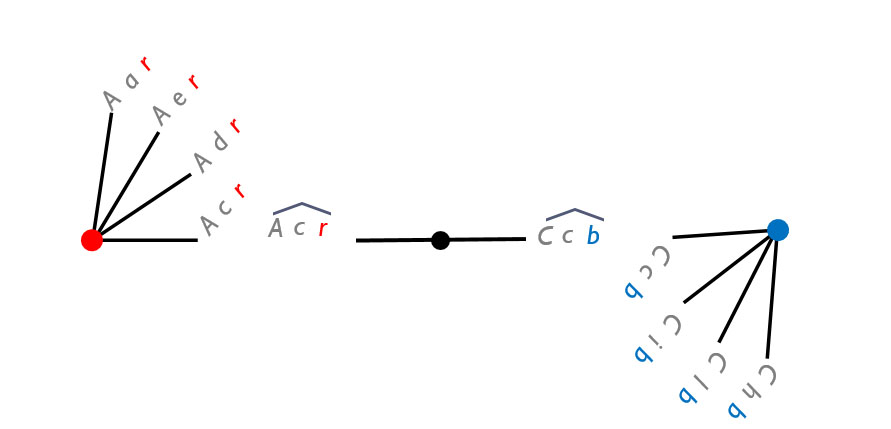} \caption{Assembly attaching vertices $A$ and $C$ with proper coloring at the edge $c$.}
      \label{ColoredVE}
      \end{figure}
We now show that a complete complex of size $k=m+n$ is realized by the pot $P$ if and only if $G$ is 3-colorable.

Suppose the graph $G$ is 3-colorable.  We begin by giving $G$ a proper coloring.  We then let $G'$ be the graph on $m+n$ vertices that results from subdividing each edge of $G$, i.e. inserting a vertex of degree two in the middle of each edge, as in Figure~\ref{subdiv}. Note that there is a proper coloring of $G$ if and only if there is a coloring of the vertices of degree four in $G'$ so that no two vertices of degree four are joined by a path of length two.

\begin{figure}
  \centering   
  \includegraphics[clip, trim=0.5cm 10cm 0.5cm 9.75cm, width=0.95\textwidth]{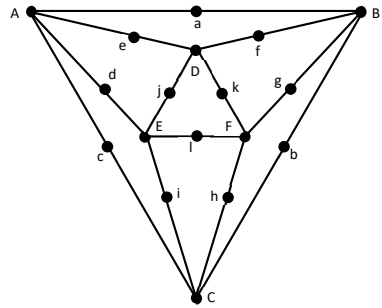}
  \caption{The graph $G'$ created by subdividing the edges of $G$.}
  \label{subdiv}
  \end{figure} 
  
For each vertex of degree four in  $G'$, we reuse the uppercase label of the corresponding vertex in $G$.  When there is an edge labeled $e$ in $G$, we also use the label $e$ for the vertex of $G'$ that arises from subdividing the edge $e$ of $G$.  Suppose that $\gamma = (e,A)$ is an edge of $G'$ (and every edge of $G'$ has this form for some edge label $e$ and vertex label $A$ from the original graph $G$ ), so that the two half-edges are $(e,\gamma)$ and $(A,\gamma)$.  Suppose further that   $A$ has color $x$ in the coloring of $G$.  Then $\lambda: (e,\gamma) \rightarrow  \widehat{Aex}$ and $\lambda: (A, \gamma) \rightarrow  Aex$ gives an assembly design for $G'$. 
%\cory{7-8-21: Should the second mapping be $\lambda: (A, \gamma) \rightarrow Aex$? Don't we want to map the half-edge $(A, \gamma)$ to the label?}

With this assembly design, if $A$ is a vertex of $G$ with incident edges $e_1, e_2, e_3, e_4$ and  with  color $x$ in the coloring of $G$, then when we view $A$ as a vertex of $G'$ we have that $\lambda(A) = \{ Ae_1x, Ae_2x, Ae_3x, Ae_4x\} $ which is an element of $P$.  By construction, if $e$ is vertex of degree two in $G'$, then $\lambda(e) \in P$.   Thus $G'$, a graph of order $k$, is realized by $P$. 

 Thus, if $G$ is 3-colorable, then $P$ realizes a complete complex of size $k$, namely $G'$.

Now suppose $P$ realizes a graph $F$ with $k=m+n$ vertices. We claim that $F$ is isomorphic to $G'$. Recall that a function $f$ is a graph homomorphism from $G'$ to $G$ if $(u,v) \in E(G') \implies (f(u), f(v)) \in E(G)$.

Let $\mu$ be an assembly design for $F$ with image in $P$. Then for any vertex $v \in F$, the arm labels on $\mu(v)$ identify a vertex label of $G$, and hence of $G'$.  We define $s:V(F) \rightarrow V(G')$ by mapping $v \in V(F) $ to the unique vertex of $G'$ with the label identified by $\mu(v)$.   We see that the map $s$ is surjective as follows.  If $v$ is any vertex of $F$ and $w'$ is any vertex of $G'$, then there must be a sequence of tiles in the assembly design of $G'$ from  $\mu(v)$ to some tile with arms labeled by the vertex label of $w'$.  This follows because $G$ is connected, so there is a path from $v'=s(v)$ to $w'$ in $G'$.  Because $F$ is is a complete complex, in $P_{\mu}(F)$ there must be a sequence of tiles, alternating between 2- and 4- armed tiles, with labels identifying, respectively, the edges and vertices along this path, irrespective of the colors on the labels.  Since the terminus of this sequence is a tile with label identifying the label of $w'$, then there is a vertex of $F$ that $s$ maps to $w$.  Thus $s$ is surjective.  Since $s$ is surjective, and $F$ is the same order as $G'$, it follows that $s$ is a bijection.

To see that $s$ is a graph homomorphism, note that if $(v, w)$ is an edge of $F$, then necessarily one of $v$ or $w$ is of degree two and the other is of degree four.  Say $v$ has degree two and $w$ has degree four.     Further suppose that $s(v) =v'$, where $v'$ has label $e$ corresponding to an edge $e$ in $G$, and $s(w)=w'$  where $w'$ has label $A$ corresponding to a vertex label in $G$.  Since $v$ and $w$ are adjacent in $F$, the tiles they map to in the assembly design have arms with labels of the form $\widehat{Axe}$  (for $v$) and $Axe$    (for $w$), for some color $x$.  But since $e$ is an edge of $G$, it also labels the corresponding vertex $v$ of $G'$. Thus $(s(v), s(w))= (v', w')$ is an edge of $G'$.

Since $s$ is a bijection between the vertex sets that preserves vertex degrees, it follows that $F$ and $G'$ have the same degree sequences, which means that they have the same number of edges.  Thus, $s$ is also a bijection on the edge sets, and hence an isomorphism.

The assembly of $F$ produces a color label for each vertex of degree four in $F$, and the bijection $s$ then assigns these colors to each vertex of degree four in $G'$.  However, each 2-armed tile in $P$ can only join two 4-armed tiles with different color labels.  Thus this coloring of of $G'$ is such that no two vertices with the same color label are joined by a path of length two. This in turn gives a proper 3-coloring of $G$. 

Thus, if $P$ realizes a complete complex of size $k$, then $G$ is 3-colorable.

This completes the reduction of the PRP to 3-coloring, and hence the proof.
\end{proof}
 Since we cannot determine if a pot realizes a graph of a given size, we cannot in general determine its spectrum, hence the following corollary.

\begin{corollary}Determining the set $\mathcal{O}(P)$ is NP-hard.  \end{corollary}

\subsection{Intractability of preventing substructures}

The following question is both pragmatically relevant and more subtle than Theorem~\ref{PRPhard}. Given a graph $G$ and a pot $P$ that realizes $G$, for example $P=P_{\lambda}(G)$ for some design $\lambda$, will the pot $P$ realize any other graphs of the same order or smaller than $G$?  This question arises when there is a design for the target structure, but if smaller structures can form from the pot, they may consume the tiles more efficiently than larger structures such as the target, and thus reduce yield.  Also, if undesired complete complexes the same size as the target are realized by the pot, it may be problematic to distinguish the target from the output of the experiment.  

We begin with the first question, namely establishing the computational complexity of determining if an assembling pot $P_{\lambda}(G)$ arising from an assembly design $\lambda$ for a graph $G$ can realize any graphs strictly smaller than $G$.  This question can also be formulated as a sort of ILP problem, albeit a somewhat unusual one.  It asks:  Given that an integer solution to $M_k(P)$ exists, does there exist an integer solution to $M_j(P)$ for $j<k$?

We show once again that this problem is intractable. Here is the decision problem. 
\begin{decisionproblem}{Substructure Realization Problem (SRP)}
 Given an assembling pot $P_{\lambda}(H)$ for a graph $H$ of order $n$ and an integer $m$ with $0<m<n$, does $P_{\lambda}(H)$ realize a graph $F$ with $m$ vertices?
\end{decisionproblem}

Again, a `yes' instance can be certified simply by checking that each element of a pot $P_{\mu}(F)$ from an assembly design $\mu$ for a graph $F$ with order less than $n$ is in $P_{\lambda}(H)$.  Thus, the \textsc{SRP} is in NP.

We will prove that the \textsc{SRP} is NP-complete here as well with a polynomial-time reduction to 3-coloring. 

\begin{theorem}
The \textsc{SRP} is NP-hard.
\end{theorem}

\begin{proof}
We show that the SRP is hard even in the special case that $n=8k$ and $m=3k$. 

The proof is in two steps.  In the first step, starting with a 4-regular planar graph $G$ of order $k$, we construct a  graph $H$ of order $8k$ in polynomial time.  We then find an assembling pot $P_{\lambda}(H)$ for it, also in polynomial time.    In the second step we show that the pot $P_{\lambda}(H)$ realizes a graph of order $3k$ if and only if $G$ is 3-colorable. 

\emph{Step 1.}  Let $G$ be a simple, connected, 4-regular, planar graph with $k$ vertices (and hence necessarily $2k$ edges).  We again label the vertices of $G$ by (for example) uppercase letters, and the edges by (for example) lower case letters.  Let $G''$ be the graph that results from subdividing all the edges of $G$ using vertices of degree two, and adding three loops to each vertex of degree four, as in Figure~\ref{looped}.

\begin{figure}
  \centering   
  \includegraphics[clip, trim=0.5cm 0cm 0.5cm 18cm, width=0.95\textwidth]{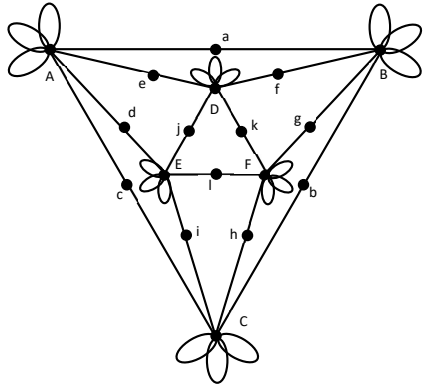}
  \caption{The graph $G''$ created by subdividing the edges of $G$ and adding three loops to each vertex of degree four.}
  \label{looped}
  \end{figure} 
  
Note $G''$ has $3k$ vertices and that a proper coloring of $G$ corresponds to a coloring of the degree ten vertices of $G''$ so that there is no path of length two between vertices of the same color.  We label the degree ten vertices of $G''$ by the corresponding uppercase letter from $G$, and the degree two vertices by the lower case letter corresponding to the subdivided edge from $G$. 

 We then create another new graph, $H$, also derived from $G$.  We begin by replacing each vertex of $G$ by the vertex `blow up' shown in Figure \ref{blowup}.  
 
 \begin{figure}
  \centering   
  \includegraphics[clip, trim=1.5cm 20cm 2.5cm 0cm, width=1\textwidth]{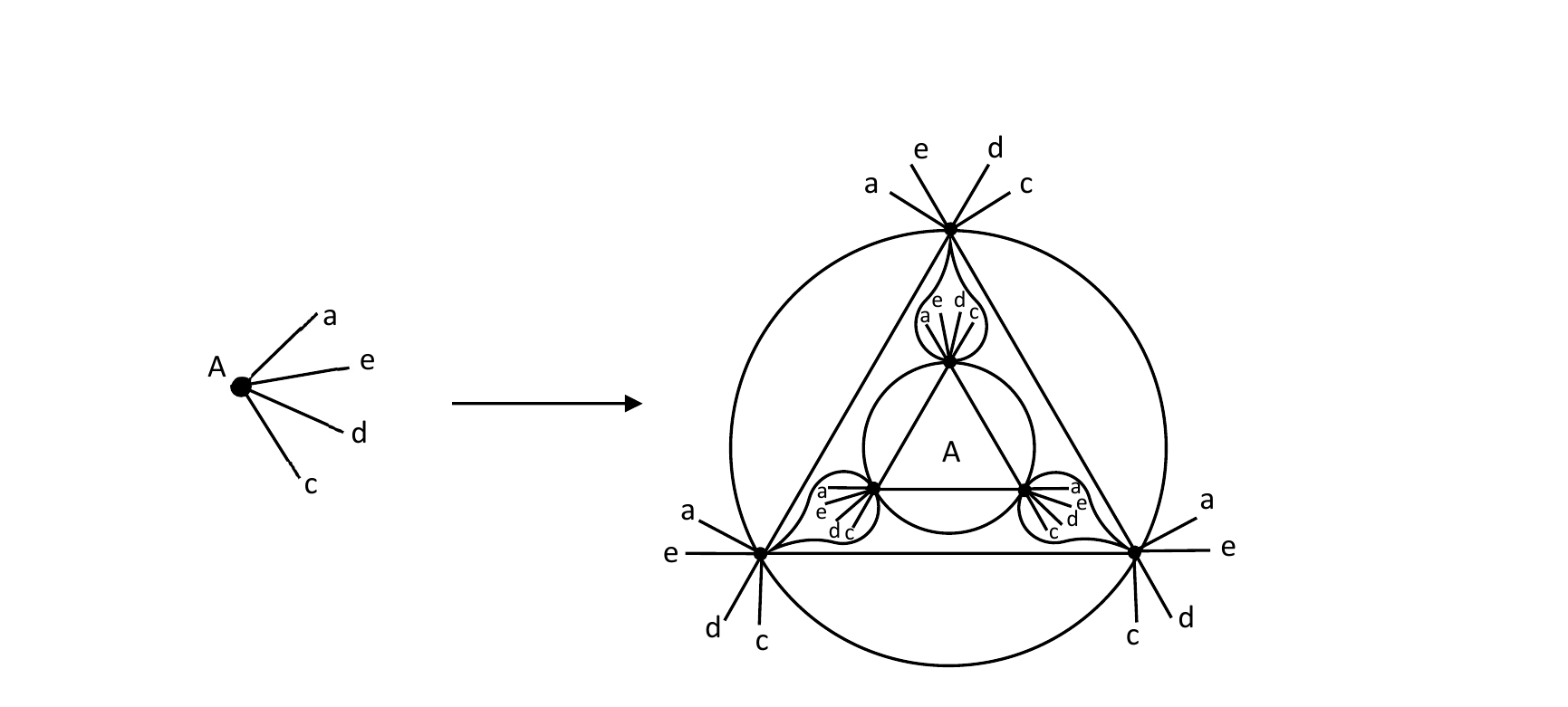}
  \caption{The `blow up' of a vertex of $G$.}
  \label{blowup}
  \end{figure}

 When there is an edge $e$ in $G$, we add six edges between the `blow ups' corresponding to the end points of $e$.  We then let the graph $H$ be the result of subdividing each each of these edges as in Figure~\ref{sixedges}. 
 
 \begin{figure}
  \centering   
  \includegraphics[clip, trim=1.5cm 10cm 2.5cm 10cm, width=1\textwidth]{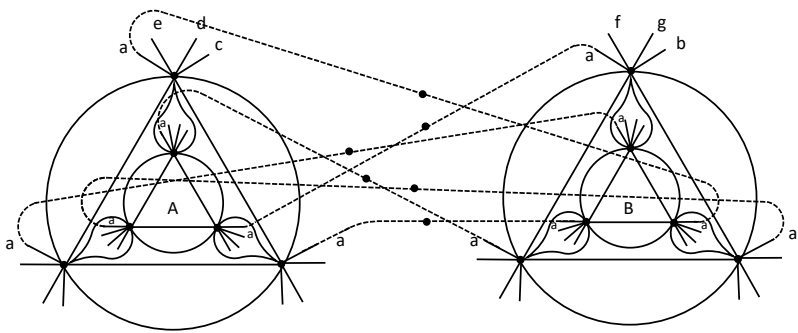}
  \caption{The six subdivided edges between the blow up of a vertex labeled A in $G$ and a blow up of a vertex labeled B.}
  \label{sixedges}
  \end{figure}

 $H$ thus has $8k$ vertices of degrees either two or ten.   
 
 The creation of $H$ clearly may be accomplished in polynomial time in the number of vertices of $G$.
 
 Then, for each vertex of degree ten, we create three 10-armed tiles, one for each color.  Four of the arms are labeled as in Theorem \ref{PRPhard}, namely by the triple of the vertex label from the original graph $G$,  the label of one of the edges incident with the vertex, and a color label.  Then, for the remaining six arms, for a tile with color label $r$, we label three of them by the original vertex label, then $br$, $gr$ or $rr$.  The other three arms receive the hatted versions of these labels.   For a tile with color label $b$, three of the six additional arms are again labeled by the original vertex label, and then by $br$, $bg$, and $bb$, and the other three by the hatted versions.  For a vertex with color label $g$, we use $gr$, $bg$, and $gg$.  See Figure~\ref{Htiles}.
  This can clearly be done in linear time in the number of vertices.   Note that the tiles here are similar to those for Theorem~\ref{PRPhard}, with the addition of three pairs of self-complementary arms added to the 4-armed tiles in Theorem~\ref{PRPhard}.   

\begin{figure}
  \centering   
  \includegraphics[clip, trim=1.5cm 2cm 2.5cm 19cm, width=1\textwidth]{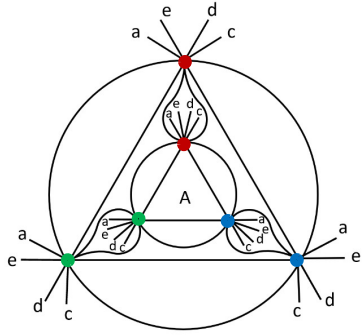}
  \caption{A blow up of a vertex $A$ incident to edges $a,e,d,c$ in $G$.  The tile corresponding to a red vertex is $\{Aar, Aer, Adr, Acr, Arr, \widehat{Arr}, Abr, \widehat{Abr}, Agr, \widehat{Agr}\}$.  The tile corresponding to a green vertex is $\{Aag, Aeg, Adg, Acg, Agg, \widehat{Agg}, Agr, \widehat{Agr}, Abg, \widehat{Abg}\}$.   The tile corresponding to a blue vertex is $\{Aab, Aeb, Adb, Acb, Abb, \widehat{Abb}, Abr, \widehat{Abr}, Abg, \widehat{Abg}\}$.}
  \label{Htiles}
  \end{figure} 

  Then, for each vertex of degree two, we create six 2-armed tiles, exactly as in Theorem~\ref{PRPhard} and Figure~\ref{EdgeTilesColored}, namely a    triple consisting of: the edge label; a vertex label so that the two arms on a tile receive the labels of the two vertices the edge is incident to; and a color label so that the two arms receive different color labels. We then take the `hat' of each triple.  Creating these tiles can be done in linear time in the number of edges.
 This establishes the assembling pot $P=P_{\lambda}(H)$.
 
\bigskip 

\emph{Step 2.}  
We now show that $P$ realizes a graph of order $m = 3k$, in particular the graph $G''$, if and only if the original graph $G$ is 3-colorable.

\emph{Part 1.  $G$ 3-colorable implies $P$ realizes $G''$.} If $G$ is 3-colorable, then we see that $P$ realizes $G''$ as follows. First fix a 3-coloring of $G$.  Suppose $A$ is a vertex with color $c$ in the labeling and coloring of $G$, and further suppose $A$ is incident with edges $e_1, e_2, e_3, e_4$.  Then an assembly design for $G''$ maps the corresponding vertex  $A$ in $G''$ to the tile  $\{Ae_1c, Ae_2c, Ae_3c, Ae_4c, Acc, \widehat{Acc}, Acc_1, \widehat{Acc_1}, Acc_2, \widehat{Acc_2}\}$.  Here $c_1$ and $c_2$ are the two colors other than $c$, and we reorder $c$ and $c_i$ in $Vcc_i$ as needed to be in alphabetical order.   

A vertex, labeled say $a$, of degree two in $G''$ is part of a path of length two between two vertices of degree ten labeled, say $A$ and $B$.  Suppose the corresponding vertices in $G$ are colored $c_A$ and $c_B$.  Then the assembly map assigns to the vertex $a$ the tile $\{\widehat{Aac_A}, \widehat{Bac_B}\}$.  

Since the three pairs of complementary arms can form loops on the 10-armed tiles, these tiles can behave as looped versions of the 4-armed tiles in Theorem~\ref{PRPhard}.  $G$ is 3-colorable by assumption, and so we have an assembly design for $G''$ as in Theorem~\ref{PRPhard}.  Since each of these tiles is in the pot $P$, the pot $P$ does realize $G''$ and hence a graph of order $3k$. 

\bigskip

\emph{Part 2. $P$ realizes $G''$ implies $G$ is 3-colorable.} Now assume $P$ realizes a graph of order $3k$, say $F$.   We claim that $F$ is isomorphic to $G''$.  

We let $\mu$ be an assembly design for $F$ with image in $P$.  We define a map $s: F \rightarrow G''$ exactly as in Theorem~\ref{PRPhard}.  The proof that $s$ is a bijection between the vertex sets follows \emph{mutis mutandis} from the proof in Theorem~\ref{PRPhard}.   Again, if we show that $s$ is a homomorphism, then since the degree sequences of $F$ and $G''$ are the same, $s$ must be an isomorphism.

For an edge of $F$ between a vertex of degree two and a vertex of degree ten, the proof that it maps to an edge of $G''$ is exactly the same as in Theorem~\ref{PRPhard}. 

Thus to complete the proof that $s$ is an isomoprhism it remains only to show that if $f$ is an edge in $F$ between two vertices of degree ten, then it maps to some edge of $G''$.  However, the only arms of 10-armed tiles that can bond to arms of another 10-armed tile have the same vertex labels (e.g. $Arr$ and $\widehat{Arr}$).  It follows that if $f$ is an edge in $F$ between two vertices of degree ten, then $s$ maps its endpoints  to the same vertex in $G''$, and thus $s$ maps $f$ to a loop on that vertex.

Thus, $P$ realizes $G''$.

However, since the three self-complementary arms on the 10-armed tiles must form loops in realizing $G''$, these tiles perform effectively as the 4-armed tiles in Theorem~\ref{PRPhard}, with two of them connected by a path of length two only if they have different colors.  Thus, if $P$ assembles $G''$ then $G''$ has no path of length two between vertices of the same color, and this exactly corresponds to a proper coloring of $G$. 
\end{proof}

We leave open the computational complexity of determining if an assembling pot for a graph $G$ can also realize a graph $H$ the same order as $G$ but not isomorphic to $G$, but conjecture that this problem is NP-hard as well.

Note that having a unique solution to the matrix of the pot does not fully resolve this question, since it is possible that the same proportions of tiles that create the target $G$ also create some other graph not isomorphic to $G$. (See for example the  assembly design for the complete graph given in Scenario 2.)
 
 Since both the PRP and the SRP are NP-hard, the design challenges now fall to finding algorithms and optimal solutions for specific situations. In the next section we provide a program for special forms of the construction matrix, and   sections following give explicit solutions for selected graphs that illustrate both general techniques and specific design difficulties. 
 
\section{Construction Matrices with Limited Degrees of Freedom}\label{sec:code}
 Although not true in general, in practice the construction matrix often has limited degrees of freedom. We leverage both this phenomenon and the special form of the construction matrix to providing a program that solves the ILP in the case of only a few degrees of freedom.
 
While rank$(M(P))-\#P$ can be arbitrarily large, the pots of tiles for the graph classes in Scenario 2 in \cite{ellis2014minimal} satisfied the property that $B_2(G) + 1 = T_2(G)$. That is, rank$(M(P)) - \# P = 0$ for each corresponding construction matrix $M(P)$, and thus the construction matrix readily verifies that no smaller graphs may be realized by the same pot of tiles.
For more complex graphs, especially lattice graphs which grow in multiple dimensions,  sometimes  rank$(M(P)) - \#P > 0$. When this occurs, the spectrum of the pot is a set of solutions with at least one free variable.  

Our program produces an output when the construction matrix contains at most two degrees of freedom. The program can readily be adapted for higher degrees of freedom, but with increasing run-time costs.  The program can be found in \cite{MapleProgram}. 

To make use of this software, a user can input the order of a target graph and the construction matrix corresponding to a pot that realizes the target graph. The output of the program will indicate if a smaller graph or graph of the same order may be created from the pot of tiles. If so, the output will contain the proportions of tile types to produce such a graph. If the solution to the construction matrix has more than two degrees of freedom, a message informs the user of this and the program terminates.

\subsection{Algorithm description}
Let $P = \{t_1, \ldots, t_p\}$ be a pot.
Recall, if $\langle r_1, \ldots, r_p \rangle \in \mathcal{S}(P)$, and if there exists $n \in \mathbb{Z}^+$ such that $nr_j \in \mathbb{Z}_{\geq 0}$ for all $j$, then there exists a graph $G \in \mathcal{O}(P)$ such that $\#V(G)=n$. In general, although a solution set with at least one free variable contains infinitely many solutions, we are only interested in solutions that describe proportions of tile types. This means for each $i$, $r_i \in \mathbb{Q}$ and $0 \leq r_i \leq 1$.  The program exploits both the restriction on the $r_i$ and an upper bound provided by the user input to test finitely many integers for $n$. In the event of one or two degrees of freedom, the program follows  mathematical processes to find possible values for the free variable(s) in which $nr_i$ is a nonnegative integer for some nonconstant $r_i$. The finitely many possibilities for the free variable(s) and corresponding values for $n$ are then checked against the remaining equations since $nr_j$ must be a nonnegative integer simultaneously for every $j$.

\subsection{Algorithm Summary}\label{sec:program}
\noindent \textbf{Input:} The construction matrix, $M(P)$, of a pot and an integer, $m$ (the order of the target graph).\\

\noindent \textbf{Output: }
\begin{enumerate}
    \item If the number of degrees of freedom of $M(P)$ is 0, then the output is the unique solution $\langle r_1, \ldots, r_p \rangle$ and the smallest positive integer $n \leq m$ such that $nr_j \in \mathbb{Z}_{\geq 0}$ for all $j$.
    \item If the number of degrees of freedom of $M(P)$ is 1 or 2, then the output is all solutions $\langle r_1, \ldots, r_p \rangle$ and corresponding positive integers $n \leq m$ such that $nr_j \in \mathbb{Z}_{\geq 0}$ for all $j$.
    \item If the number of degrees of freedom of $M(P)$ is 3 or more, then a message is displayed stating as much.
\end{enumerate}

\noindent \textbf{Algorithm:}
\begin{enumerate}
\item Compute $RREF(M(P))$ and calculate the number of free variables.
    \item If there are 0 free variables, calculate the LCD of the solution vector $ \langle r_1, \ldots, r_p \rangle$; return the solution vector if LCD $\leq m$ and a message else.
     \item If there is 1 free variable:
        \begin{enumerate}
            \item Express solution to system of equations in parametric form $r_i(t) = a_it + c_i$ where $a_i, c_i \in \mathbb{Q}$.
           
            \item For each constant equation, $r_i(t)$, find $N_i = \{n_{i_j} \in \mathbb{Z} \: | \: 1 \leq n_{i_j} \leq m \text{ and } n_{i_j} \cdot r_i(t) \in \mathbb{Z}_{\geq 0}  \}$ and let $S = \displaystyle \bigcap_i N_i$.
          
            \item Let $r_k(t)$ be the first nonconstant equation. For each $n \in S$ and for each integer $y$ where $1 \leq y \leq m$, solve the equation $n \cdot r_k(t) = y$. If $t \in [0,1]$ (necessary since $t$ is a proportion), then store the pairs $n$ and $t$ in lists $N$ and $T$, respectively.
            
            \item For each  pair $n \in N$ and $t \in T$, check if $n \cdot r_i(t) \in \mathbb{Z}_{\geq 0}$ for all $i > k$. Only values of $n$ and $t$ which satisfy this condition are saved for the next iteration.
            \item For each $t \in T$, compute the solution vector $\langle r_1, \ldots, r_p \rangle$. Output the solution vector and the corresponding value of $n$. 
        \end{enumerate}
    \item If there are 2 free variables, the process is similar.
        \begin{enumerate}
          
            \item Express solutions to linear system in parametric form as $r_i(t,u) = a_i t + b_i u + c_i$ where $a_i, b_i, c_i \in \mathbb{Q}$.
         
            \item For each constant equation, find $N_i = \{n_{i_j} \in \mathbb{Z} \: | \: 1 \leq n_{i_j} \leq m \text{ and } n_{i_j} \cdot r_i(t) \in \mathbb{Z}_{\geq 0} \}$ and let $S = \displaystyle \bigcap_i N_i$.
            \item Let $r_k(t,u)$ and $r_{k+1}(t,u)$ be the first two nonconstant equations. For each $n \in S$ and for every pair of integers $y_1, y_2$ such that $1 \leq y_1, y_2 \leq m$, solve the system of equations $n \cdot r_k(t,u) = y_1$ and $n \cdot r_{k+1}(t,u) = y_2$. If $(t,u) \in [0,1] \times [0,1]$, then store the triples $n, t, u$ in the lists $N, T, U$, respectively. 
            \item For each corresponding triple $n \in N, t \in T, u \in U$, check if $n \cdot r_i(t,u) \in \mathbb{Z}_{\geq 0}$ for each $i > k+1$. Only triples $(n,t,u)$ which satisfy this condition are stored for the next iteration.
            \item For each $(t,u) \in T \times U$, compute the solution vector $\langle r_1, \ldots, r_p \rangle$. Output the solution vector and the corresponding value of $n$. 
        \end{enumerate}
\end{enumerate}

This algorithm can clearly be extended to higher degrees of freedom, but with increasingly greater runtime costs.

\section{Selected Results with Proofs}\label{sec:SelectedResults}
 In this section we include a collection of results illustrating the difficulty of finding optimal pots for selected graphs  in each of the three restrictive scenarios.  Each of the examples demonstrates one or another particular design challenge.  We state and prove optimal pots for the cube graph  for all three scenarios in Section \ref{sec:cube}. Scenario 3 requires several preliminary lemmas, and the final outcome demonstrates the need for two different pots to realize the minimum number of tiles and the minimum number of bond-edge types. This answers a question posed in \cite{ellis2014minimal} in proving that it is not always possible to achieve both $B_3$ and $T_3$ in the same pot. In Section \ref{sec:squarelattice}, we provide Scenario 1 solutions for square lattices of any size, as well as an example of a pot $P$ realizing a small square lattice in which two different solutions from $\mathcal{S}(P)$ can realize isomorphic graphs and one solution from $\mathcal{S}(P)$ can realize two non-isomorphic graphs. This example demonstrates the utility of our \emph{Maple} code in achieving optimal designs. In Section \ref{sec:trianglelattice} we examine a small triangular lattice graph for which we prove it is not possible to design a pot realizing the minimum number of tile types and minimum number of bond-edge types simultaneously, demonstrating that the cube graph is not the unique graph for which this occurs. Section \ref{sec:tubes} provides results for all square and triangular lattice tube graphs in Scenario 1.  

 We summarize our entire collection of results for platonic solids, lattices, and tubes in Figure \ref{tableofresults}. Results and corresponding proofs not given here  are provided in a separate repository of proofs \cite{repository}. Note that $B_1(G)=1$ for all graphs, $T_1(G)$ was previously known for all $k$-regular graphs, and all values for complete graphs were previously known \cite{ellis2014minimal}. These values are included in the table for the reader's convenience.

\begin{figure}[htbp] \centering
\begin{tabularx}{0.9\textwidth}{|l|X|X|X|}
\hline \textbf{Graph Type} & \textbf{Scenario 1} & \textbf{Scenario 2} & \textbf{Scenario 3} \\ \hline 
Tetrahedron ($K_4)$ \footnotemark & $B_1(G) = 1 \newline T_1(G)=2$ & $B_2(G)=1$ \newline $T_2(G)=2$ & $B_3(G)=3 \newline T_3(G)=4$  \\ \hline
Hexahedron & $B_1(G)=1$ \newline $T_1(G)=2$  & $B_2(G)=2$ \newline $T_2(G)=3$ & $B_3(G)=5$ \newline $T_3(G)=6$ \\ \hline
Octahedron & $B_1(G)=1$ \newline $T_1(G)=1$ & $B_2(G)=2$ \newline $T_2(G)=3$ & $B_3(G)=4$ \newline $T_3(G)=5$ \\ \hline 
Icosahedron & $B_1(G)=1$ \newline $T_2(G)=2$ & $B_2(G)=2$ \newline $T_2(G)=3$ & $B_3(G)=9$ \newline $T_3(G)=12$ \\ \hline
Dodecahedron & $B_1(G)=1$ \newline $T_1(G)=2$ & $B_2(G) \leq 4$ \newline $T_2(G) \leq 6$ & $B_3(G) \geq 10$ \newline $T_3(G)= 20$ \\ \hline
Square Lattice & $B_1(G) = 1$ \newline $T_1(G) = 3$ or $4$ \footnotemark & $B_2(G)=2$ \newline $T_2(G) = 4$ \footnotemark & $B_3(G) = 3$ \newline $T_3(G) = 4$ \footnotemark[4] \\ \hline
Triangle Lattice & $B_1(G)=1$ \newline $T_1(G)=4$ or $5$ \footnotemark & & \\ \hline
Square Lattice Tube & $B_1(G)=1$ \newline $T_1(G)=3$ & & \\ \hline
Triangle Lattice Tube & $B_1(G)=1$ \newline $T_1(G)=1$ or $2$ \footnotemark & & \\ \hline
\end{tabularx}
\caption{Known results for Platonic Solids, Lattice Graphs, Tube Lattice Graphs}
\label{tableofresults}
\end{figure}
\footnotetext[2]{Results for $K_4$ previously known \cite{ellis2014minimal}.}
\footnotetext[3]{$T_1(G)$ value depends on %$\#V(G)$
dimensions of $G$.}
\footnotetext[4]{These $B_i(G)$ and $T_i(G)$ values are for the $2 \times 3$ lattice only.}
\footnotetext[5]{$T_1(G)$ value depends on %$\#V(G)$
dimensions of $G$.}
\footnotetext[6]{$T_1(G)$ value depends on %$\#V(G)$
dimensions of $G$.}

 We begin with the  following two lemmas which are of general utility in the remainder of our work, since they address structural constraints that arise in Scenarios 1 and 3. The first lemma asserts that $T_1$ and $B_1$ can always be achieved by the same pot, and thus  increasing the size of $\Sigma(P)$ will not reduce the size of $P$.
 
\begin{lemma}\label{onebondlemma}
Given any graph $G$, there exists a pot $P$ such that $G \in \mathcal{O}(P)$, with $\# \Sigma(P) = 1$, and $\#P = T_1(G)$. 
\end{lemma}

\begin{proof}
 Suppose there exists a pot $P'$ where $G \in \mathcal{O}(P')$, and $\# P' = T_1(G)$.  Suppose $\#\Sigma(P')>1$, with $a_1, a_2 \in \Sigma(P')$. Then, every instance of $a_2$ may be replaced by $a_1$ and $\widehat{a_2}$ may be replaced by $\widehat{a_1}$. This process may be repeated for each $a_i \neq a_1$ until we obtain a pot $P$ such that $\# \Sigma (P) =1$. Then $G \in \mathcal{O}(P)$ as well, and $T_1(G) \leq \#P$ while $\# \Sigma (P) = 1$.  
\end{proof}

 The next result shows that incident half-edges can always be labeled with the  same cohesive-end type.  Recall that we cannot have half-edges labeled $a$ and $\hat{a}$ on the same vertex in Scenario 3 without the possibility that a non-isomorphic graph can be realized if those half-edges were to disconnect and re-join in a different way. However, identical labeling of edges amounts to a ``swapping'' of the two edges, and the graph structure remains unchanged. 

\begin{lemma}\label{lemma:edgeswap}
 If $P$ realizes $G$ and if two incident edges $e_1=\{u,v\}$ and $e_2=\{u,w\}$ are labeled with the same bond-edge type with matching cohesive-end type orientation, then there is an isomorphism $\phi: E(G) \to E(G')$ where $\phi(e_1) = e_2$, $\phi(e_2) = e_1$ and $\phi(e_i) = e_i$ for all $i \neq 1,2$.
\end{lemma}

\subsection{Cube Graph and the Challenges of Scenario 3} \label{sec:cube}

The platonic solids are natural choices for study in this context, as they both represent mathematically standard graphs and also model the polyhedral cage structures assembled in laboratories \cite{zhang2012}. Complete results in all scenarios were only previously known for the tetrahedron, $K_4$. We provide a collection of results for the other platonic solids in \cite{repository}, but here we focus on the hexahedron, or the cube. The cube is of special importance, as the production of a self-assembled DNA cube served as an initial motivator for much of the existing DNA self-assembly research \cite{Seeman82}.

Since the cube is 3-regular, from \cite{ellis2014minimal} we know that $B_1(G)=1$ and $T_1(G)=2$.  Finding $B_2(G)=1$ and $T_2(G)=2$  below is straightforward and demonstrates applications of the construction matrix.  However, the challenges of Scenario 3 are particularly illustrated in this example, as Scenario 3 involves checking for and prohibiting the formation of graphs on eight vertices that are not isomorphic to the cube.

Recall that Proposition~\ref{spectrum} specifies how the spectrum of a pot identifies the sizes of the graphs it realizes. In order to emphasize the sizes of graphs realized by a given pot, in the following we will generally express the spectrum of a pot as an integer vector, with a prefactor having as a term in its denominator the minimum value specified in Item 3 of Proposition~\ref{spectrum}.

\begin{proposition}
If $G$ is the cube graph, then $B_2(G)=2$ and $T_2(G)=3$, and these values are achieved simultaneously by the same pot.
\end{proposition}

\begin{proof}
%\smartqed
 Suppose $B_2(G)=1$. Since $G$ is 3-regular, it follows that $\# P \geq 2$ \cite{ellis2014minimal}. There are only four tile types for a 3-regular graph using one bond-edge type.  In order to avoid constructing a pot $P$ such that there exists $G' \in \mathcal{O}(P)$ with $\#V(G')=2$, the only possible pot, without loss of generality, is $P' = \{t_1=\{a^2,\hat{a}\}, t_2=\{\hat{a}^3\}\}$.  The  spectrum of the pot  is $\mathcal{S}(P') = \left\{\frac{1}{4r}\langle 3r, r \rangle \: | \: r \in \mathbb{Z}^+ \right\}$. $\mathcal{S}(P')$ shows there exists $G' \in \mathcal{O}(P')$ such that $\#V(G')=4$. Thus, $B_2(G) \geq 2$.  
 
 The pot $P=\left\{\{a,b^2\},\{a^2,\hat{b}\},\{a,\hat{a}^2\}\right\}$, with two bond-edge types and three tile types, realizes the graph of the cube (see Figure \ref{fig:Cube_scenario2}). 
 
 \begin{figure}[h]
 \centering
    \includegraphics[scale=.75]{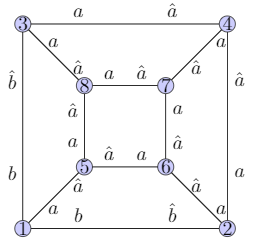}
     \caption{Assembly design of cube in Scenario 2}
     \label{fig:Cube_scenario2}
 \end{figure}
 
 The spectrum of the pot, $\mathcal{S}(P) = \left\{\frac{1}{8r} \langle r,  2r, 4r \rangle \: | \: r \in \mathbb{Z}^+ \right\}$, shows that $\#V(H) \geq 8$ for all $H \in \mathcal{O}(P).$ Thus, this pot satisfies the restrictions of Scenario 2 for the cube. 
%\qed
 \end{proof}
 
 Let $G$ denote the cube graph with vertices $v_1,...,v_8$ as labeled in Figure \ref{fig:cubeVertices}. Let $t_n$ denote the tile corresponding to a vertex labeled $n$ for $n\in\{1, 2, \ldots, 8 \}$.   
 
For clarity, when we label edge $e = \{v_i,v_j\}$ with bond-edge type $a$, then we will use the directed edge $(v_i, v_j)$ notation to mean $\lambda(v_i, e) = a$ and $\lambda(v_j,e) = \hat{a}$. 

\begin{figure}[h]
	\centering
 \includegraphics[scale=.7]{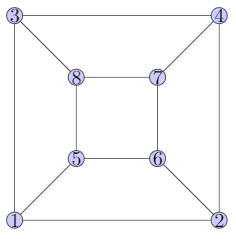}
      \caption{Cube vertex labeling}
      \label{fig:cubeVertices}
  \end{figure}

The following lemmas provide assembly design restrictions for the cube in Scenario 3. 

\begin{lemma}\label{lemma_CubeRestrictionRepeatBondType}
 For the cube $G$ in Scenario 3, suppose directed edge $e = (v_i, v_j)$ is labeled with bond-edge type $a$. Then bond-edge type $a$ may only be repeated on edges $(v_i, v_k)$, $(v_k, v_j)$, or on edge $(v_k, v_m)$ where $d(v_i,v_m) = d(v_j,v_k) = 3$.
\end{lemma}

\begin{proof}
 Let $G$ be the cube with vertex labels as in Figure \ref{fig:cubeVertices}. Assume without loss of generality that directed edge $(v_1, v_3)$ is labeled with bond-edge type $a$.  The directed edges $(v_1,v_2), (v_1, v_5), (v_4,v_3)$, or $(v_8, v_3)$ may be labeled with $a$  by Lemma \ref{lemma:edgeswap}. It can be verified that if directed edge $(v_6, v_7)$ is labeled with bond-edge type $a$, then breaking and reattaching the edges will result in an isomorphic graph. Directed edges $(v_2, v_1), (v_5,v_1), (v_3,v_4),$ and $(v_3,v_8)$ cannot be labeled $a$ since a vertex cannot have both a hatted and unhatted cohesive-end of the same type  by Lemma 2 in \cite{ellis2014minimal}. 
 
 If directed edges $(v_6, v_2),$ $(v_4, v_2),$ $(v_4, v_7),$ $(v_6,v_5),$ $(v_8,v_5),$ or $(v_8,v_7)$ are labeled with $a$, then when the edges break and reattach, the resulting graph contains a multi-edge and is nonisomorphic to $G$.  Furthermore, if directed edges $(v_2, v_6), (v_7, v_4), (v_5,v_6),  (v_7, v_6),$  or $(v_7,v_8)$ are labeled with $a$, then when the edges break and reattach, the resulting graph contains a 3-cycle and is nonisomorphic to $G$. If directed edges $(v_2, v_4)$ or $(v_5,v_8)$ are labeled with $a$, then when the edges break and reattach, the resulting graph contains a 5-cycle and is nonisomorphic to $G$.
\end{proof}

\begin{lemma}\label{lemma2_Cube3bonds}
  Let $G$ be the cube graph. In Scenario 3, if $G \in \mathcal{O}(P)$, with $t \in P$, and $\lambda(v_i) = \lambda(v_j) = t$ for some $i \neq j$, then $t= \{a, b, c\}$ where $a,b,c$ are distinct bond-edge types.
 \end{lemma}
 
 \begin{proof} 
  Let $G$ denote the cube graph with vertices $v_1,...,v_8$ as labeled in Figure \ref{fig:cubeVertices}. We show that a tile with a repeated bond-edge type may not be repeated in the assembly design of $G$.   Assume $G \in \mathcal{O}(P)$ and $t \in P$. Without loss of generality, let $\lambda(v_1)=t=\{a^2,b\}$. Furthermore, suppose $\lambda(v_1,\{v_1,v_3\}) = a$. 
  Suppose $\lambda(v_j) = t$  for $j \in \{2,...,8\}$. By Lemma 3 in \cite{ellis2014minimal}, $t$ may not be repeated on adjacent vertices. Hence, we only need to consider the cases $j=4,6,7, 8$. As in the proof of Lemma \ref{lemma_CubeRestrictionRepeatBondType}, if $\lambda(v_j) = t$ for $j \in \{4, 6, 7, 8\}$, then a graph with a multi-edge or 3-cycle may be realized by $P$.  
  Hence, a tile with a repeated bond-edge type may not be repeated in the construction of the cube in Scenario 3.
 \end{proof}

    \begin{lemma}\label{lemma1_Cube3}
Let $G$ be the cube graph. In Scenario 3, if $G \in \mathcal{O}(P)$, and $t \in P$, then $\lambda(v_i)=t$ for at most two distinct $v_i$.
\end{lemma}

\begin{proof}
We  claim that a tile cannot be repeated more than two times in the assembly design of $G$. Given the labeling of the cube vertices in Figure \ref{fig:cubeVertices}, assume $\lambda(v_1) = t$. Suppose $\lambda(v_i)= \lambda(v_j)=t$ where $i \neq j$ and $i,j \in \{2, 3, \ldots, 8\}$. From \cite{ellis2014minimal}, $i, j \in \{4, 6, 8\}$. Suppose without loss of generality $\lambda(v_4) = \lambda(v_6)= t$. Then a graph with multi-edges between $v_1$ and $v_5$, $v_4$ and $v_7$, or $v_1$ and $v_2$ may be realized.  Therefore, a tile can be repeated at most twice. 
\end{proof}

\begin{lemma}\label{lemma3_CubeNoRepeat3tiles}
Let $G$ be the cube graph. In Scenario 3, if $G \in \mathcal{O}(P)$, then $\lambda(v_i)=\lambda(v_j)$ for at most two disjoint sets of vertex pairs $\{v_i,v_j\}$.
\end{lemma}

\begin{proof}
  Without loss of generality, assume tiles $t_1, t_2,$ and $t_3$ are each repeated two times in the assembly design of $G$, and $\lambda(v_1) = t_1$.  By Lemma \ref{lemma2_Cube3bonds}, $t_1 = \{a, b, c\}$, where $a$, $b$, and $c$ are distinct bond-edge types. Furthermore, assume without loss of generality, $\lambda(v_1,\{v_1,v_3\}) = a,$     $\lambda(v_1,\{v_1,v_5\})=b,$  $\lambda(v_1,\{v_1,v_2\})= c$ (see Figure \ref{fig:CubeRepeat_t1}).  Then since tile types may not be repeated on adjacent vertices, $\lambda(v_k) = t_1$ only if $k = 4, 6, 7,$ or 8.

 Note that if $\lambda(v_7) = t_1$, then bond-edge type $a$ must appear on either directed edges $(v_7,v_6), (v_7,v_8),$ or $(v_7,v_4)$ which contradicts Lemma \ref{lemma_CubeRestrictionRepeatBondType}. Thus $k = 4, 6,$ or $8$, but by the symmetry of the cube, these three cases are equivalent.

 Suppose $\lambda(v_1) = \lambda(v_8) = t_1 = \{a, b, c\}$. By Lemma \ref{lemma_CubeRestrictionRepeatBondType}, the three edges incident to $v_8$ can only be labeled in one way (see Figure \ref{fig:CubeRepeat_t1}). 
 
 \begin{figure}[h]
     \centering
     \includegraphics[scale=.7]{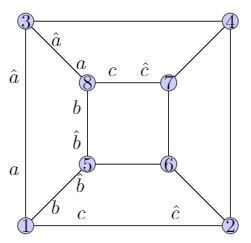}
     \caption{Repeating $t_1$ on two vertices}
     \label{fig:CubeRepeat_t1}
 \end{figure}
 
 Because two edges at $v_3$ are labeled with $\hat{a}$ and two edges at $v_5$ are labeled with $\hat{b}$, then the tile types at $v_3$ and $v_5$ cannot appear more than once by Lemma \ref{lemma2_Cube3bonds}. Furthermore, since tile types cannot be repeated on adjacent vertices, then it must be the case that $\lambda(v_2) = \lambda(v_7) = t_2$ and $\lambda(v_4)=\lambda(v_6) = t_3$. Now consider $t_2$. Since $d(v_1,v_6) \neq 3$ and $d(v_1,v_4) \neq 3$, then $a \notin t_2$ and $b \notin t_2$ by Lemma \ref{lemma_CubeRestrictionRepeatBondType}. Thus, $t_2 = \{\hat{c}, d, e\}$ and $\hat{d}, \hat{e} \in t_3$ (see Figure \ref{fig:CubeRepeat_t1_t2}). Since $\lambda(v_7)= t_2$, then $\lambda(v_7,\{v_7,v_4\}) = d$ or $e$. Given that $\lambda(v_4) = t_3$ and $\lambda(v_4,\{v_4,v_2\})=\hat{e}$ then $\lambda(v_7,\{v_7,v_4\}) \neq e$ by Lemma \ref{lemma2_Cube3bonds}.  It follows from Lemma \ref{lemma_CubeRestrictionRepeatBondType} that $\lambda(v_7,\{v_7,v_4\}) \neq d$ since $\lambda(v_2,\{v_2,v_6\}) = d$ and $d(v_7,v_6) \neq 3$. 
  
 \begin{figure}[h]
     \centering
    \includegraphics[scale=.7]{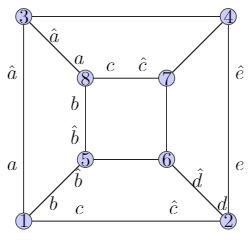}
    \caption{Repeating $t_1, t_2, t_3$}
     \label{fig:CubeRepeat_t1_t2}
 \end{figure}
 
 Therefore, three tile types may not be repeated two times each.
\end{proof}

 \begin{proposition}\label{prop:CubeT3}
 Let $G$ be the cube. Then $T_3(G)=6$.
 \end{proposition} 

 \begin{proof}

 Suppose $T_3(G) \leq 5$. Then either $R_i \geq 3$ for some $i$ or $R_i = 2, R_j = 2, R_k = 2$ for some $i \neq j \neq k$. In other words, a tile must appear at least three times or three tiles must appear two times in the assembly design of $G$.  By Lemma \ref{lemma1_Cube3} and   Lemma \ref{lemma3_CubeNoRepeat3tiles}, this is not possible. Thus, $T_3(G)\geq 6$. Consider the pot $P = \left\{\{a,b,c\}, \{\hat{a}^2,\hat{e}\}, \{e, d, f\}, \{\hat{b}, \hat{d}^2\}, \{\hat{c}^2, \hat{e}\},\{\hat{b}, \hat{f}^2\} \right\}$. Figure \ref{fig:Cube13} shows that $G \in \mathcal{O}(P)$. Solutions to the construction matrix show that $\# V(H) \geq 8$ for all $H \in \mathcal{O}(P)$, since $\mathcal{S}(P)=\left\{\frac{1}{8r} \langle 2r,r,2r,r,r,r \rangle | r \in \mathbb{Z}\right\}$. Lemma \ref{lemma_CubeRestrictionRepeatBondType} guarantees that if any edges break and re-attach the graph formed is isomorphic to $G$. %\qed
\end{proof}

       \begin{figure}[h!]
      \centering
      \includegraphics[scale=.7]{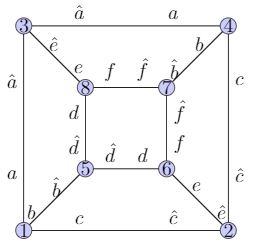}
      \caption{Pot realizing $G$ in Proposition \ref{prop:CubeT3}}
      \label{fig:Cube13}
  \end{figure}

\begin{proposition}\label{prop:CubeB3}
Let $G$ be the cube. Then $B_3(G)=5$.
\end{proposition}

\begin{proof}
Assume $B_3(G) \leq 4$. Then either at least one bond-edge type appears more than three times or four bond-edge types must each appear exactly three times. By Lemma \ref{lemma_CubeRestrictionRepeatBondType} only the latter is possible, and repeated bond-edge types must appear on incident edges. In this case the only possible labeling of $G$ results in the following pot: $$P' = \left\{\{a^3\}, \{b^3\}, \{c^3\}, \{d^3\}, \{\hat{a}, \hat{b}, \hat{c} \}, \{\hat{a}, \hat{b}, \hat{d} \}, \{\hat{a}, \hat{c}, \hat{d} \}, \{\hat{b}, \hat{c}, \hat{d} \}\right\}.$$ However, there exists a graph $ H \in \mathcal{O}(P')$ such that $\# V(H) = 6$, namely $K_{(3,3)}$ using tiles $\{a^3\}, \{b^3\}, \{c^3\}$, and $\{\hat{a}, \hat{b}, \hat{c} \}$.

Thus $B_3(G) \geq 5$. The following pot $P$ with $\#\Sigma(P) = 5$ realizes the cube as shown in Figure \ref{fig:CubeB3}:
$$P=\left\{\{a^3\}, \{e^3\}, \{b^2,\hat{a}\}, \{c^2,\hat{b}\}, \{d^2,\hat{b}\},\{\hat{a},\hat{c},\hat{e}\},\{\hat{c},\hat{d},\hat{e}\},\{\hat{a},\hat{d},\hat{e}\}\right\}.$$ The spectrum $\mathcal{S}(P)=\left\{\frac{1}{8r} \langle r,r,r,r,r,r, r, r \rangle | r \in \mathbb{Z}\right\}$ shows that $\# V(H) \geq 8$ for all $H \in \mathcal{O}(P)$. Because the only repeated bond-edge types are on edges $(v_i,v_j)$ and $(v_k, v_m)$ satisfying $d(v_i,v_j)=(v_k, v_m)=3$, then as in the argument for Lemma \ref{lemma_CubeRestrictionRepeatBondType}, if any edges break and re-attach the resulting graph is isomorphic to $G$. 
\end{proof}
     \begin{figure}
      \centering
      \includegraphics[scale=.7]{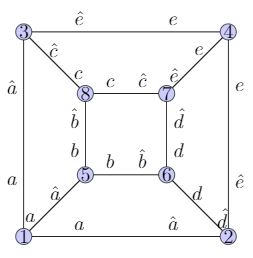}
      \caption{Pot realizing the cube in Proposition \ref{prop:CubeB3}}
      \label{fig:CubeB3}
  \end{figure}

Note that the cube graph is an example in which the pots that achieve the minimum number of tile types and the minimum number of bond-edge types are different. Whether or not there exists a pot that can achieve both minima simultaneously remains an open question; however, we conjecture that no such pot exists.

\subsection{Square Lattice Graphs and Isomorphism Limitations of the Construction Matrix} \label{sec:squarelattice}

We study square and triangular lattice graphs and tube graphs with DNA meshes and mesh tubes in mind. Many of the examples of lattice and tube graphs we provide also illustrate the difficulty in determining optimal solutions, even in the least restrictive scenario. In this section, we provide Scenario 1 pots for all square lattices. We conclude with an example where non-isomorphic graphs of the same size are realized by the same pot, thus illustrating a limitation of the construction matrix. 

\begin{definition} A \textit{lattice graph}, or \textit{grid graph}, is a graph whose embedding in $\mathbb{R}^2$ forms a regular tiling. \end{definition}

\begin{definition} The $m \times n$ \emph{square lattice graph} is the graph Cartesian product $P_m \times P_n$ of path graphs on $m$ and $n$ vertices, where $\# V(G) = mn$ and the number of 4-cycles is  $(m-1)(n-1)$. \end{definition}

In \cite{ellis2014minimal}, it was shown that for all graphs, $av(G) \leq T_1(G) \leq ov(G) + 2ev(G)$, where $av(G)$ is the length of the valency sequence of $G$, $ov(G)$ is the length of the odd valency sequence of $G$, and $ev(G)$ is the length of the even valency sequence of $G$. For $m \times n$ square lattice graphs with $m>2$ and $n>2$, $3 \leq T_1(G) \leq 4$. For $2 \times n$ or $m \times 2$ square lattice graphs, $2 \leq T_1(G) \leq 3$. Note that the $2 \times 2$ square lattice graph is simply $C_4$ and we do not consider this graph as a proper square lattice. The results below find explicit pots of tiles that will realize a square lattice graphs of any size.

\begin{proposition}\label{props1mostlattices} Let $G$ be a square lattice graph of size $m \times n$ with $m,n >2$ and $\{m,n\} \not \in \{\{3,3\}, \{3,5\}, \{4,4\}\}$. Then $T_1(G)=4$. \end{proposition} 
\begin{proof} 
 Assume $G \in \mathcal{O}(P)$ with $\#P = 3$. Let $t_i \in P$ denote the $i+1$-armed tile. The following equation must be satisfied:
\begin{equation} \label{eq:sqlattice1} 4z_{1,1}+2(m+n-4)z_{1,2}+(m-2)(n-2)z_{1,3} = 0. \end{equation}
Note that if $m>3$ or $n>3$ then at least one pair of  vertices of degree three will be adjacent and at least one pair of vertices of degree four will be adjacent in $G$, so $z_{1,2} \in \{\pm 1\}$  and $z_{1,3} \in \{0,\pm 2\}$. Since $t_1$ is a  2-armed tile, $z_{1,1} \in \{0,\pm 2\}$ and $4z_{1,1} \in \{0,\pm 8\}$.\par
In order for Equation \ref{eq:sqlattice1} to hold, $2(m+n-4)z_{1,2}+(m-2)(n-2)z_{1,3} \in \{0, \pm 8\}$. \par
If $2(n+m-4)z_{1,2}+(m-2)(n-2)z_{1,3} = 0$, then $z_{1,3} \neq 0$, hence $|(n+m-4)|=|(m-2)(n-2)|$. The only integer solution pair to this equation with $m>2, n>2$ is $(4,4)$. If $2(n+m-4)z_{1,2}+(m-2)(n-2)z_{1,3} = \pm 8$, the integer solution pairs with $m>2, n>2$ are $(3,5),(4,4)$, and $(4,8)$.  It is straightforward to verify that any pot satisfying Equation \ref{eq:sqlattice1} can not realize the $4 \times 8$ square lattice graph. \par
The pot $P = \left\{\{a,\hat{a}\},\{a,\hat{a}^2\},\{a^2,\hat{a}\},\{a^2,\hat{a}^2\}\right\}$ realizes all square lattice graphs of size $m \times n$ with $m, n > 2$ and $\{m,n\} \not \in \{\{3,3\}, \{3,5\}, \{4,4\}\}$.
\end{proof} 

\begin{example}
Realization of the pot given in Proposition \ref{props1mostlattices} for the $5 \times 5$ square lattice graph is shown in Figure \ref{fig:5x5lattice}.

\begin{figure}[h] \centering
\begin{tikzpicture}[transform shape, scale =.85]
        %bottom side:
        \draw(0,0)--node[below]{$a$}(1,0);
        \draw(1,0)--node[below]{$\hat{a}$}(2,0);
        \draw(2,0)--node[below]{$\hat{a}$}(3,0);
        \draw(3,0)--node[below]{$a$}(4,0);
        \draw(4,0)--node[below]{$a$}(5,0);
        \draw(5,0)--node[below]{$\hat{a}$}(6,0);
        \draw(6,0)--node[below]{$\hat{a}$}(7,0);
        \draw(7,0)--node[below]{$a$}(8,0);
        %left side:
        \draw(0,0)--node[left]{$\hat{a}$}(0,1);
        \draw(0,1)--node[left]{$a$}(0,2); 
        \draw(0,2)--node[left]{$a$}(0,3);
        \draw(0,3)--node[left]{$\hat{a}$}(0,4);
        \draw(0,4)--node[left]{$\hat{a}$}(0,5);
        \draw(0,5)--node[left]{$a$}(0,6);
        \draw(0,6)--node[left]{$a$}(0,7);
        \draw(0,7)--node[left]{$\hat{a}$}(0,8);
        %top side:
        \draw(0,8)--node[above]{$a$}(1,8);
        \draw(1,8)--node[above]{$\hat{a}$}(2,8);
        \draw(2,8)--node[above]{$\hat{a}$}(3,8);
        \draw(3,8)--node[above]{$a$}(4,8);
        \draw(4,8)--node[above]{$a$}(5,8);
        \draw(5,8)--node[above]{$\hat{a}$}(6,8);
        \draw(6,8)--node[above]{$\hat{a}$}(7,8);
        \draw(7,8)--node[above]{$a$}(8,8);
        %right side:
        \draw(8,0)--node[right]{$\hat{a}$}(8,1);
        \draw(8,1)--node[right]{$a$}(8,2);
        \draw(8,2)--node[right]{$a$}(8,3);
        \draw(8,3)--node[right]{$\hat{a}$}(8,4);
        \draw(8,4)--node[right]{$\hat{a}$}(8,5);
        \draw(8,5)--node[right]{$a$}(8,6);
        \draw(8,6)--node[right]{$a$}(8,7);
        \draw(8,7)--node[right]{$\hat{a}$}(8,8);
        %interior bottom row:
        \draw(2,0)--node[left]{$a$}(2,1);
        \draw(2,1)--node[left]{$\hat{a}$}(2,2);
        \draw(4,0)--node[left]{$\hat{a}$}(4,1);
        \draw(4,1)--node[left]{$a$}(4,2);
        \draw(0,2)--node[below]{$\hat{a}$}(1,2);
        \draw(1,2)--node[above]{$a$}(2,2);
        \draw(2,2)--node[below]{$a$}(3,2);
        \draw(3,2)--node[above]{$\hat{a}$}(4,2);
        \draw(4,2)--node[below]{$\hat{a}$}(5,2);
        \draw(5,2)--node[above]{$a$}(6,2);
        \draw(6,0)--node[left]{$a$}(6,1);
        \draw(6,1)--node[left]{$\hat{a}$}(6,2);
        \draw(6,2)--node[below]{$a$}(7,2);
        \draw(7,2)--node[below]{$\hat{a}$}(8,2);
        %interior second row:
        \draw(2,2)--node[right]{$\hat{a}$}(2,3);
        \draw(2,3)--node[left]{$a$}(2,4);
        \draw(4,2)--node[right]{$a$}(4,3);
        \draw(4,3)--node[left]{$\hat{a}$}(4,4);
        \draw(0,4)--node[below]{$a$}(1,4);
        \draw(1,4)--node[above]{$\hat{a}$}(2,4);
        \draw(2,4)--node[below]{$\hat{a}$}(3,4);
        \draw(3,4)--node[above]{$a$}(4,4);
        \draw(4,4)--node[below]{$a$}(5,4);
        \draw(5,4)--node[above]{$\hat{a}$}(6,4);
        \draw(6,2)--node[right]{$\hat{a}$}(6,3);
        \draw(6,3)--node[left]{$a$}(6,4);
        \draw(6,4)--node[below]{$\hat{a}$}(7,4);
        \draw(7,4)--node[below]{$a$}(8,4);
        %interior third row:
        \draw(2,4)--node[right]{$a$}(2,5);
        \draw(2,5)--node[left]{$\hat{a}$}(2,6);
        \draw(4,4)--node[right]{$\hat{a}$}(4,5);
        \draw(4,5)--node[left]{$a$}(4,6);
        \draw(6,4)--node[right]{$a$}(6,5);
        \draw(6,5)--node[left]{$\hat{a}$}(6,6);
        \draw(6,6)--node[below]{$a$}(7,6);
        \draw(7,6)--node[below]{$\hat{a}$}(8,6);
        %interior top row:
        \draw(0,6)--node[above]{$\hat{a}$}(1,6);
        \draw(1,6)--node[above]{$a$}(2,6);
        \draw(2,6)--node[below]{$a$}(3,6);
        \draw(3,6)--node[above]{$\hat{a}$}(4,6);
        \draw(4,6)--node[below]{$\hat{a}$}(5,6);
        \draw(5,6)--node[above]{$a$}(6,6);
        \draw(2,6)--node[right]{$\hat{a}$}(2,7);
        \draw(2,7)--node[right]{$a$}(2,8);
        \draw(4,6)--node[right]{$a$}(4,7);
        \draw(4,7)--node[right]{$\hat{a}$}(4,8);
        \draw(6,6)--node[right]{$\hat{a}$}(6,7);
        \draw(6,7)--node[right]{$a$}(6,8);
        
        \filldraw (0,0) circle [radius=3pt];
        \filldraw (0,2) circle [radius=3pt];
        \filldraw (0,4) circle [radius=3pt];
        \filldraw (0,6) circle [radius=3pt];
        \filldraw (0,8) circle [radius=3pt];
        \filldraw (2,0) circle [radius=3pt];
        \filldraw (2,2) circle [radius=3pt];
        \filldraw (2,4) circle [radius=3pt];
        \filldraw (2,6) circle [radius=3pt];
        \filldraw (2,8) circle [radius=3pt];
        \filldraw (4,0) circle [radius=3pt];
        \filldraw (4,2) circle [radius=3pt];
        \filldraw (4,4) circle [radius=3pt];
        \filldraw (4,6) circle [radius=3pt];
        \filldraw (4,8) circle [radius=3pt];
        \filldraw (6,0) circle [radius=3pt];
        \filldraw (6,2) circle [radius=3pt];
        \filldraw (6,4) circle [radius=3pt];
        \filldraw (6,6) circle [radius=3pt];
        \filldraw (6,8) circle [radius=3pt];
        \filldraw (8,0) circle [radius=3pt];
        \filldraw (8,2) circle [radius=3pt];
        \filldraw (8,4) circle [radius=3pt];
        \filldraw (8,6) circle [radius=3pt];
        \filldraw (8,8) circle [radius=3pt];
        \end{tikzpicture} \caption{Scenario 1 assembly design of $5 \times 5$ square lattice graph}\label{fig:5x5lattice} \end{figure}
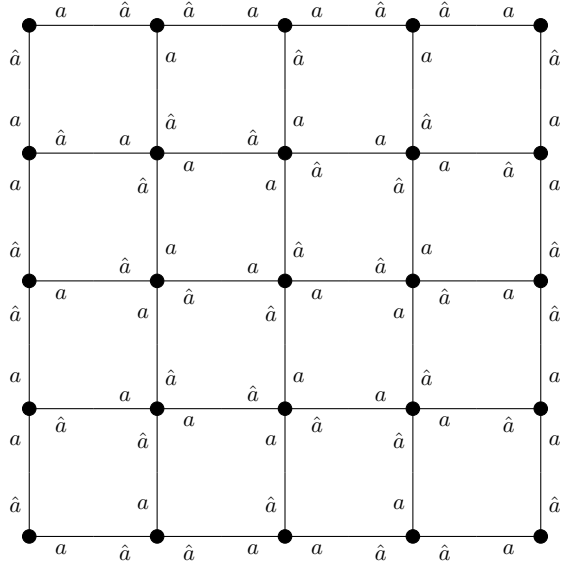 
\end{example}

\begin{proposition} If $G$ is a square lattice graph of size $3 \times 3$, $3 \times 5$, or $4 \times 4$ then $T_1(G)=3$. \end{proposition} 
\begin{proof} 
$G \in \mathcal{O}(P)$ where $P = \left\{ \{ a^2 \}, \{a, \hat{a}^2\}, \{a^2, \hat{a}^2\} \right\}$. Realizations of the pot $P$ for the $4 \times 4$ and $3 \times 5$ square lattices are illustrated in Figures \ref{fig:4x4lattice} and \ref{fig:3x5lattice}, respectively. The square lattice graph of size $3 \times 3$ is also known as the gear graph $G_4$. Detailed results for this graph are given in \cite{Mattamira}. 
\end{proof} 

 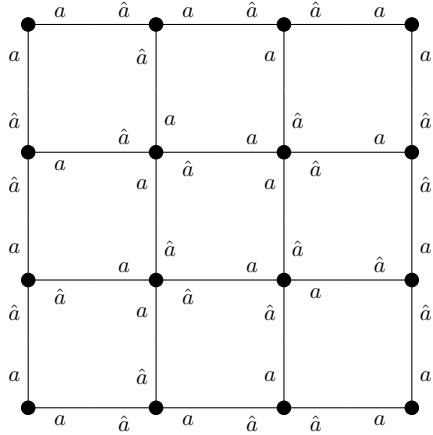
\begin{figure}[h] \centering
\begin{tikzpicture}[transform shape, scale=.85]
        %bottom side:
        \draw(0,0)--node[below]{$a$}(1,0);
        \draw(1,0)--node[below]{$\hat{a}$}(2,0);
        \draw(2,0)--node[below]{$a$}(3,0);
        \draw(3,0)--node[below]{$\hat{a}$}(4,0);
        \draw(4,0)--node[below]{$\hat{a}$}(5,0);
        \draw(5,0)--node[below]{$a$}(6,0);
        %left side:
        \draw(0,0)--node[left]{$a$}(0,1);
        \draw(0,1)--node[left]{$\hat{a}$}(0,2); 
        \draw(0,2)--node[left]{$a$}(0,3);
        \draw(0,3)--node[left]{$\hat{a}$}(0,4);
        \draw(0,4)--node[left]{$\hat{a}$}(0,5);
        \draw(0,5)--node[left]{$a$}(0,6);
        %top side:
        \draw(0,6)--node[above]{$a$}(1,6);
        \draw(1,6)--node[above]{$\hat{a}$}(2,6);
        \draw(2,6)--node[above]{$a$}(3,6);
        \draw(3,6)--node[above]{$\hat{a}$}(4,6);
        \draw(4,6)--node[above]{$\hat{a}$}(5,6);
        \draw(5,6)--node[above]{$a$}(6,6);
        %right side:
        \draw(6,0)--node[right]{$a$}(6,1);
        \draw(6,1)--node[right]{$\hat{a}$}(6,2);
        \draw(6,2)--node[right]{$a$}(6,3);
        \draw(6,3)--node[right]{$\hat{a}$}(6,4);
        \draw(6,4)--node[right]{$\hat{a}$}(6,5);
        \draw(6,5)--node[right]{$a$}(6,6);
        %interior bottom row:
        \draw(2,0)--node[left]{$\hat{a}$}(2,1);
        \draw(2,1)--node[left]{$a$}(2,2);
        \draw(4,0)--node[left]{$a$}(4,1);
        \draw(4,1)--node[left]{$\hat{a}$}(4,2);
        \draw(0,2)--node[below]{$\hat{a}$}(1,2);
        \draw(1,2)--node[above]{$a$}(2,2);
        \draw(2,2)--node[below]{$\hat{a}$}(3,2);
        \draw(3,2)--node[above]{$a$}(4,2);
        \draw(4,2)--node[below]{$a$}(5,2);
        \draw(5,2)--node[above]{$\hat{a}$}(6,2);
        %interior middle row:
        \draw(2,2)--node[right]{$\hat{a}$}(2,3);
        \draw(2,3)--node[left]{$a$}(2,4);
        \draw(4,2)--node[right]{$\hat{a}$}(4,3);
        \draw(4,3)--node[left]{$a$}(4,4);
        \draw(0,4)--node[below]{$a$}(1,4);
        \draw(1,4)--node[above]{$\hat{a}$}(2,4);
        \draw(2,4)--node[below]{$\hat{a}$}(3,4);
        \draw(3,4)--node[above]{$a$}(4,4);
        \draw(4,4)--node[below]{$\hat{a}$}(5,4);
        \draw(5,4)--node[above]{$a$}(6,4);
        %interior top row:
        \draw(2,4)--node[right]{$a$}(2,5);
        \draw(2,5)--node[left]{$\hat{a}$}(2,6);
        \draw(4,4)--node[right]{$\hat{a}$}(4,5);
        \draw(4,5)--node[left]{$a$}(4,6);
        
        \filldraw (0,0) circle [radius=3pt];
        \filldraw (0,2) circle [radius=3pt];
        \filldraw (0,4) circle [radius=3pt];
        \filldraw (2,0) circle [radius=3pt];
        \filldraw (2,2) circle [radius=3pt];
        \filldraw (2,4) circle [radius=3pt];
        \filldraw (4,0) circle [radius=3pt];
        \filldraw (4,2) circle [radius=3pt];
        \filldraw (4,4) circle [radius=3pt];
        \filldraw (6,0) circle [radius=3pt];
        \filldraw (6,2) circle [radius=3pt];
        \filldraw (6,4) circle [radius=3pt];
        \filldraw (0,6) circle [radius=3pt];
        \filldraw (2,6) circle [radius=3pt];
        \filldraw (4,6) circle [radius=3pt];
        \filldraw (6,6) circle [radius=3pt];
        \end{tikzpicture} \caption{Scenario 1 assembly design of $4 \times 4$ square lattice graph}\label{fig:4x4lattice} \end{figure} 
\begin{figure}[h]  \centering
\begin{tikzpicture}[transform shape, scale=.85]
        %bottom side:
        \draw(0,0)--node[below]{$a$}(1,0);
        \draw(1,0)--node[below]{$\hat{a}$}(2,0);
        \draw(2,0)--node[below]{$\hat{a}$}(3,0);
        \draw(3,0)--node[below]{$a$}(4,0);
        \draw(4,0)--node[below]{$\hat{a}$}(5,0);
        \draw(5,0)--node[below]{$a$}(6,0);
        \draw(6,0)--node[below]{$\hat{a}$}(7,0);
        \draw(7,0)--node[below]{$a$}(8,0);
        %left side:
        \draw(0,0)--node[left]{$a$}(0,1);
        \draw(0,1)--node[left]{$\hat{a}$}(0,2); 
        \draw(0,2)--node[left]{$\hat{a}$}(0,3);
        \draw(0,3)--node[left]{$a$}(0,4);
        %right side:
        \draw(8,0)--node[right]{$a$}(8,1);
        \draw(8,1)--node[right]{$\hat{a}$}(8,2);
        \draw(8,2)--node[right]{$\hat{a}$}(8,3);
        \draw(8,3)--node[right]{$a$}(8,4);
        %interior bottom row:
        \draw(2,0)--node[left]{$a$}(2,1);
        \draw(2,1)--node[left]{$\hat{a}$}(2,2);
        \draw(4,0)--node[left]{$\hat{a}$}(4,1);
        \draw(4,1)--node[left]{$a$}(4,2);
        \draw(0,2)--node[below]{$a$}(1,2);
        \draw(1,2)--node[above]{$\hat{a}$}(2,2);
        \draw(2,2)--node[below]{$a$}(3,2);
        \draw(3,2)--node[above]{$\hat{a}$}(4,2);
        \draw(4,2)--node[below]{$a$}(5,2);
        \draw(5,2)--node[above]{$\hat{a}$}(6,2);
        \draw(6,0)--node[left]{$\hat{a}$}(6,1);
        \draw(6,1)--node[left]{$a$}(6,2);
        \draw(6,2)--node[below]{$\hat{a}$}(7,2);
        \draw(7,2)--node[above]{$a$}(8,2);
        %interior second row:
        \draw(2,2)--node[right]{$a$}(2,3);
        \draw(2,3)--node[left]{$\hat{a}$}(2,4);
        \draw(4,2)--node[right]{$\hat{a}$}(4,3);
        \draw(4,3)--node[left]{$a$}(4,4);
        \draw(0,4)--node[above]{$a$}(1,4);
        \draw(1,4)--node[above]{$\hat{a}$}(2,4);
        \draw(2,4)--node[above]{$a$}(3,4);
        \draw(3,4)--node[above]{$\hat{a}$}(4,4);
        \draw(4,4)--node[above]{$\hat{a}$}(5,4);
        \draw(5,4)--node[above]{$a$}(6,4);
        \draw(6,2)--node[right]{$a$}(6,3);
        \draw(6,3)--node[left]{$\hat{a}$}(6,4);
        \draw(6,4)--node[above]{$\hat{a}$}(7,4);
        \draw(7,4)--node[above]{$a$}(8,4);
   
        \filldraw (0,0) circle [radius=3pt];
        \filldraw (0,2) circle [radius=3pt];
        \filldraw (0,4) circle [radius=3pt];
        \filldraw (2,0) circle [radius=3pt];
        \filldraw (2,2) circle [radius=3pt];
        \filldraw (2,4) circle [radius=3pt];
        \filldraw (4,0) circle [radius=3pt];
        \filldraw (4,2) circle [radius=3pt];
        \filldraw (4,4) circle [radius=3pt];
        \filldraw (6,0) circle [radius=3pt];
        \filldraw (6,2) circle [radius=3pt];
        \filldraw (6,4) circle [radius=3pt];
        \filldraw (8,0) circle [radius=3pt];
        \filldraw (8,2) circle [radius=3pt];
        \filldraw (8,4) circle [radius=3pt];
        \end{tikzpicture} \caption{Scenario 1 assembly design of $3 \times 5$ square lattice graph}\label{fig:3x5lattice} \end{figure}
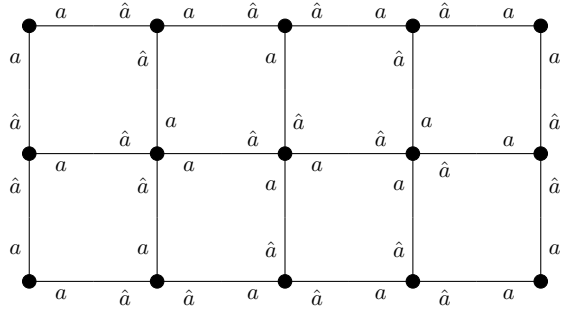 
 
The following proposition is specifically for $2 \times n$ square lattice graphs, which are sometimes referred to as ``ladder'' graphs.
\begin{proposition} Let $G$ be a square lattice graph of size $2 \times n$ with $n>2$, then $T_1(G)=3$. \end{proposition} 
\begin{proof}
Recall the valency sequence of $G$ is $(2,3)$, so $2 \leq T_1(G) \leq 3$ \cite{ellis2014minimal}.  Let $G \in \mathcal{O}(P)$ with $\#P = 2$. Note that the  $2$-armed tile must be $\{a,\hat{a}\}$ since degree two vertices are adjacent to each other in $G$. There are $2(n-2)$ degree three vertices in $G$, so $2(n-2)z_{1,2}=0$, where $z_{1,2}$ is the net number of cohesive-ends on the 3-armed tile type. This forces $z_{1,2}=0$, which is impossible for a 3-armed tile with a single bond-edge type. Lemma \ref{onebondlemma} ensures that $T_1(G)$ can be achieved with a pot $P$ such that $\# \Sigma P = 1$. Thus, $T_1(G) = 3$, and $G$ is realized by the pot $P = \left\{ \{ a, \hat{a} \}, \{ a^2, \hat{a} \}, \{a, \hat{a}^2\} \right\}.$
\end{proof} 

In Scenario 2, the construction matrix can be used to identify the order of the smallest graph realized by a pot of tiles, as well as the proportion of tiles to use to construct the graph. However, it is possible for a construction matrix to have multiple solutions. These solutions may or may not correspond to isomorphic graphs. Furthermore, a single solution may realize multiple non-isomorphic graphs. The following examples demonstrate the limitations of the solutions obtained from the construction matrix.

\begin{example}\label{limits_of_matrix_ex}
It can be easily verified that the pot $P = \left\{ \{a,b\}, \{a,\hat{b}\}, \{\hat{a}^2,b\}, \right.$ $\left.\{\hat{a}^2,\hat{b}\}\right\}$ realizes the $2 \times 3$ square lattice graph. The construction matrix and spectrum of the pot are shown below.

\begin{equation}
    M(P)=\begin{amatrix}{4} 1 & 1 & -2 & -2 & 0 \\ 1 & -1 &1 &-1 & 0\\1 & 1 & 1 & 1 & 1  \end{amatrix}
\end{equation}

\begin{equation} \mathcal{S}(P) = \left\{ \frac{1}{6r} \langle r+t, 3r-t, 2r-t, t \rangle \middle| \; r \in \mathbb{Z}^+,  t \in \left(\mathbb{Z}\cap [-r,2r]\right) \right\} \end{equation}

 Note that there is a free variable in the solution set for $M(P)$. This creates difficulty in assessing by hand the smallest ordered graph in $\mathcal{O}(P)$. Using the program in Section \ref{sec:program}, it can be shown that the order of the smallest graph that can be realized from this pot of tiles is six with $r=1$ and $t \in \{ 0, 1, 2\}$. When $t = 0$, then one graph that can be realized is shown in Figure \ref{fig:r0graph}. When $t = 1$, two non-isomorphic graphs can be realized, including the $2 \times 3$ square lattice, as shown in Figures \ref{fig:2x3lattice} and \ref{fig:2x3latticenonisom}. Note that the graphs in Figures \ref{fig:r0graph} and \ref{fig:2x3latticenonisom} are the same but are constructed using different tile proportions.

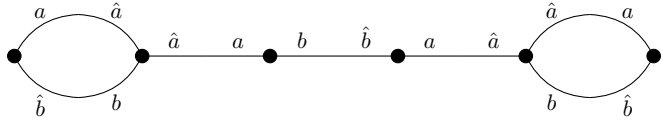
\begin{figure}  \centering
\begin{tikzpicture}[transform shape, scale=.85]
        %bottom side:
        \draw (0,0) [bend left] to node [above] {$a$} (1,0.65);
        \draw(1,0.65) [bend left] to node[above]{$\hat{a}$}(2,0);
        \draw (0,0) [bend right] to node [below] {$\hat{b}$} (1,-0.65);
        \draw(1,-0.65) [bend right] to node[below]{$b$}(2,0);
        \draw(2,0)--node[above]{$\hat{a}$}(3,0);
        \draw(3,0)--node[above]{$a$}(4,0);
        \draw(4,0)--node[above]{$b$}(5,0);
        \draw(5,0)--node[above]{$\hat{b}$}(6,0);
        \draw(6,0)--node[above]{$a$}(7,0);
        \draw(7,0)--node[above]{$\hat{a}$}(8,0);
        \draw (8,0) [bend left] to node [above] {$\hat{a}$} (9,0.65);
        \draw(9,0.65) [bend left] to node[above]{$a$}(10,0);
        \draw (8,0) [bend right] to node [below] {$b$} (9,-0.65);
        \draw(9,-0.65) [bend right] to node[below]{$\hat{b}$}(10,0);
        
        \filldraw (0,0) circle [radius=3pt];
        \filldraw (2,0) circle [radius=3pt];
        \filldraw (4,0) circle [radius=3pt];
        \filldraw (6,0) circle [radius=3pt];
        \filldraw (8,0) circle [radius=3pt];
        \filldraw (10,0) circle [radius=3pt];

        \end{tikzpicture} \caption{Graph formed from $P$ when $r=1$ and $t=0$ }
        \label{fig:r0graph}
        \end{figure}

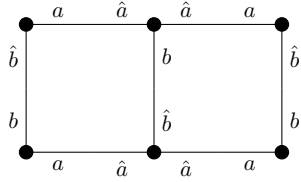
\begin{figure} \centering
\begin{tikzpicture}[transform shape, scale = .85]
        %bottom side:
        \draw(0,0)--node[below]{$a$}(1,0);
        \draw(1,0)--node[below]{$\hat{a}$}(2,0);
        \draw(2,0)--node[below]{$\hat{a}$}(3,0);
        \draw(3,0)--node[below]{$a$}(4,0);
        %left side:
        \draw(0,0)--node[left]{$b$}(0,1);
        \draw(0,1)--node[left]{$\hat{b}$}(0,2); 
        %interior bottom row:
        \draw(2,0)--node[right]{$\hat{b}$}(2,1);
        \draw(2,1)--node[right]{$b$}(2,2);
        \draw(4,0)--node[right]{$b$}(4,1);
        \draw(4,1)--node[right]{$\hat{b}$}(4,2);
        \draw(0,2)--node[above]{$a$}(1,2);
        \draw(1,2)--node[above]{$\hat{a}$}(2,2);
        \draw(2,2)--node[above]{$\hat{a}$}(3,2);
        \draw(3,2)--node[above]{$a$}(4,2);
        
        \filldraw (0,0) circle [radius=3pt];
        \filldraw (0,2) circle [radius=3pt];
        \filldraw (2,0) circle [radius=3pt];
        \filldraw (2,2) circle [radius=3pt];
        \filldraw (4,0) circle [radius=3pt];
        \filldraw (4,2) circle [radius=3pt];
        \end{tikzpicture}\caption{$2 \times 3$ square lattice graph formed from $P$ when $r=1$ and $t=1$} \label{fig:2x3lattice} \end{figure}
        
        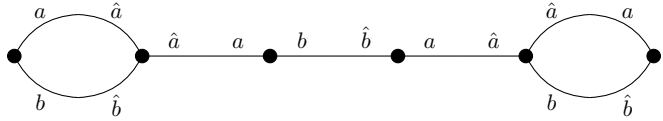
\begin{figure}[h!] \centering
        \begin{tikzpicture}[transform shape, scale =.85]
        %bottom side:
        \draw (0,0) [bend left] to node [above] {$a$} (1,0.65);
        \draw(1,0.65) [bend left] to node[above]{$\hat{a}$}(2,0);
        \draw (0,0) [bend right] to node [below] {$b$} (1,-0.65);
        \draw(1,-0.65) [bend right] to node[below]{$\hat{b}$}(2,0);
        \draw(2,0)--node[above]{$\hat{a}$}(3,0);
        \draw(3,0)--node[above]{$a$}(4,0);
        \draw(4,0)--node[above]{$b$}(5,0);
        \draw(5,0)--node[above]{$\hat{b}$}(6,0);
        \draw(6,0)--node[above]{$a$}(7,0);
        \draw(7,0)--node[above]{$\hat{a}$}(8,0);
        \draw (8,0) [bend left] to node [above] {$\hat{a}$} (9,0.65);
        \draw(9,0.65) [bend left] to node[above]{$a$}(10,0);
        \draw (8,0) [bend right] to node [below] {$b$} (9,-0.65);
        \draw(9,-0.65) [bend right] to node[below]{$\hat{b}$}(10,0);

        \filldraw (0,0) circle [radius=3pt];
        \filldraw (2,0) circle [radius=3pt];
        \filldraw (4,0) circle [radius=3pt];
        \filldraw (6,0) circle [radius=3pt];
        \filldraw (8,0) circle [radius=3pt];
        \filldraw (10,0) circle [radius=3pt];
        \end{tikzpicture}
        \caption{Non-isomorphic graph formed from $P$ when $r=1$ and $t=1$} \label{fig:2x3latticenonisom} \end{figure}

\end{example}

Example \ref{limits_of_matrix_ex} demonstrates that given one pot of tiles, two isomorphic graphs may be realized using a different ratio of tiles. The example also shows that the same ratio of tiles may realize non-isomorphic graphs. 
\begin{remark} For the $2 \times 3$ square lattice graph $G$, $T_2(G)=4, B_2(G)=2, T_3(G) = 4$, and $B_3(G) = 3$. Proofs are provided in \cite{repository}.  \end{remark}

\subsection{Triangle Lattice Graphs and No Pot Simultaneously Achieving $B_3$ and $T_3$ } \label{sec:trianglelattice}
\begin{definition} The $m \times n$ \emph{triangular lattice graph} is the graph Cartesian product $P_m \times P_n$ of path graphs on $m$ and $n$ vertices with diagonal edges from the bottom left to the top right of each square. That is, if the vertices of the square lattice are integer Cartesian coordinates, then the diagonal edges are in between coordinates $(i,j)$ and $(i+1, j+1)$. The number of 3-cycles is $2(m-1)(n-1)$ and $\#V(G) = mn$. %lattice graph  consisting of $m-1$ rows of 3-cycles and $n$ 3-cycles in each row.
\end{definition}

 An interesting question when finding $T_i(G)$ and $B_i(G)$ for $i= 2, 3$ is whether both minima can be achieved simultaneously. With the exception of the cube, both minima have been achieved for families of graphs presented in the literature. The following example shows that there is no pot that achieves both the minimum number of tiles and the minimum number of bond-edge types needed to realize the $2 \times 3$ triangular lattice graph in Scenario 3.
 
%\begin{figure}\centering\includegraphics[scale=.65]{4Triangle.JPG}\caption{$2 \times 4$ Triangle Lattice Graph}\label{TriangleExample}\end{figure}

\begin{figure}[h] \centering
\begin{tikzpicture}[transform shape, scale = .85,
  arrow inside/.style = {
    postaction={decorate},
    decoration={markings, mark=at position 0.95 with {\arrow[scale=2]{stealth}}}
  }
  ]
        %bottom side:
        \draw(0,0)--node[below]{}(1,0);
        \draw(1,0)--node[below]{}(2,0);
        \draw(2,0)--node[below]{}(3,0);
        \draw(3,0)--node[below]{}(4,0);
        %diagonals first row:
        \draw(0,0)--node[below]{}(1,1);
        \draw(1,1)--node[below]{}(2,2);
         \draw(2,0)--node[below]{}(3,1);
        \draw(3,1)--node[below]{}(4,2);
        %left side:
        \draw(0,0)--node[left]{}(0,1);
        \draw(0,1)--node[left]{}(0,2); 

        %interior bottom row:
        \draw(2,0)--node[right]{}(2,1);
        \draw(2,1)--node[right]{}(2,2);
        \draw(4,0)--node[right]{}(4,1);
        \draw(4,1)--node[right]{}(4,2);
        %top side:
        \draw(0,2)--node[above]{}(1,2);
        \draw(1,2)--node[above]{}(2,2);
        \draw(2,2)--node[above]{}(3,2);
        \draw(3,2)--node[above]{}(4,2);
        
        \node [Node Label Style] at (0,0) {$1$};
        \node [Node Label Style] at (2,0) {$3$};
        \node [Node Label Style] at (4,0) {$5$};
        \node [Node Label Style] at (0,2) {$2$};
        \node [Node Label Style] at (2,2) {$4$};
        \node [Node Label Style] at (4,2) {$6$};

        \end{tikzpicture} \caption{$2 \times 3$ triangle lattice graph}\label{TriangleExample} \end{figure}
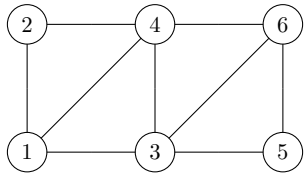

\begin{example}

Consider the $2\times 3$ triangular lattice graph, $G$, as shown in Figure \ref{TriangleExample}. By  \cite{Almodovar2019}, $T_3(G) = 4$ and $B_3(G) =3$.

Now we show that given a pot $P$ that uses only three bond-edge types, it is impossible to realize $G$ with four tile types. By way of contradiction, suppose that $\#P =4$ and $\#\Sigma(P)=3$. Since the degree four vertices $v_3$ and $v_4$ are adjacent, $
\lambda(v_3) \neq \lambda(v_4)$ by Lemma 3 in \cite{ellis2014minimal}. Hence $P=\{t_1, t_2, t_3, t_4\}$ where $\lambda(v_1)=\lambda(v_6)=t_1$, $\lambda(v_2)=\lambda(v_5)=t_2$, $\lambda(v_3)=t_3$ and $\lambda(v_4)=t_4$. Without loss of generality, there are only two choices for $t_2$, either $t_2=\{a,a\}$ or $t_2=\{a,b\}$. If $t_2 = \{a,a\}$ then $P$ can realize a graph in which a degree two vertex is adjacent to two degree three vertices, which is not isomorphic to $G$ as shown in Figure \ref{TriangleNonIsoExample}. Hence, $t_2=\{a,b\}$. Without loss of generality, assume $\lambda(v_2, \{v_2,v_4\}) = a$. As in the argument when $t_2 = \{a,a\}$, we have that $\lambda(v_5,\{v_5,v_6\}) \neq a$, otherwise a non-isomorphic graph may be realized as shown in Figure \ref{TriangleNonIsoExample}. For the remainder of this proof we assume, as above, $\lambda(v_2, \{v_2,v_4\}) = a$, $\lambda(v_5,\{v_5,v_3\}) = a$, $\lambda(v_2,\{v_2,v_1\}) = b$, and $\lambda(v_5,\{v_5,v_6\}) = b$.

Consider the two remaining arms of $t_1$. If either of these arms is labeled with $\hat{a}$ or $b$, a graph with a loop can be realized. If either of the arms are labeled with $a$ or $\hat{b}$ a graph with  multiple edges can be realized. Thus, we must label both of the arms with a new bond-edge type, $c$. Notice that with this labeling  multiple edges can form between vertices $v_1$ and $v_4$ or between $v_4$ and $v_6$, thus creating a non-isomorphic graph. This shows there exists $H \in \mathcal{O}(P)$ such that $H$ is not isomorphic to $G$ when $\# P =4$ and $\# \Sigma(P) =3$. \par
Scenario 1 solutions for selected dimensions of triangle lattice graphs are provided in \cite{repository, Almodovar2019}. For all triangle lattice graphs of the type described here, $4 \leq T_1(G) \leq 5$ \cite{repository}.

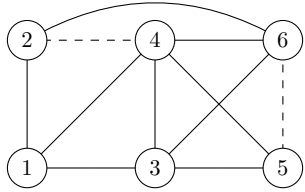
\begin{figure}[h] \centering
\begin{tikzpicture}[transform shape, scale = .85,
  arrow inside/.style = {
    postaction={decorate},
    decoration={markings, mark=at position 0.95 with {\arrow[scale=2]{stealth}}}
  }
  ]
        %bottom side:
        \draw(0,0)--node[below]{}(1,0);
        \draw(1,0)--node[below]{}(2,0);
        \draw(2,0)--node[below]{}(3,0);
        \draw(3,0)--node[below]{}(4,0);
        %diagonals first row:
        \draw(0,0)--node[below]{}(1,1);
        \draw(1,1)--node[below]{}(2,2);
         \draw(2,0)--node[below]{}(3,1);
        \draw(3,1)--node[below]{}(4,2);
        %left side:
        \draw(0,0)--node[left]{}(0,1);
        \draw(0,1)--node[left]{}(0,2); 

        %interior bottom row:
        \draw(2,0)--node[right]{}(2,1);
        \draw(2,1)--node[right]{}(2,2);
        \draw[dashed](4,0)--node[right]{}(4,1);
        \draw[dashed](4,1)--node[right]{}(4,2);
        %top side:
        \draw[dashed](0,2)--node[above]{}(1,2);
        \draw[dashed](1,2)--node[above]{}(2,2);
        \draw(2,2)--node[above]{}(3,2);
        \draw(3,2)--node[above]{}(4,2);
        
        \path[draw](0,2) edge [bend left] node []{} (4,2);
         \path[draw](2,2) edge [] node []{} (4,0);
        
        \node [Node Label Style] at (0,0) {$1$};
        \node [Node Label Style] at (2,0) {$3$};
        \node [Node Label Style] at (4,0) {$5$};
        \node [Node Label Style] at (0,2) {$2$};
        \node [Node Label Style] at (2,2) {$4$};
        \node [Node Label Style] at (4,2) {$6$};

        \end{tikzpicture} \caption{$2 \times 3$ triangle lattice graph}\label{TriangleNonIsoExample} \end{figure}

 \end{example}

\subsection{ Triangle and Square Lattice Tube Graphs in Scenario 1}\label{sec:tubes}
Fortunately, there are classes and settings for which provably optimal strategies are readily attainable.  Lattice tube graphs represent an abstraction of the construction of DNA tubes formed from meshes. In this section we give systematic design strategies for triangle and square lattice tube graphs in Scenario~1. 

\begin{definition} A \emph{lattice tube} is a lattice graph in which one pair of opposite edges at the ends of each row of the lattice have been identified, allowing the lattice to wrap and form a tube structure in 3-space. \end{definition} 
 
 Since triangle lattice tubes in the case that $n \leq 3$ are degenerate, we will assume that $n \geq 4$. The $2 \times n$ case is 4-regular, so from \cite{ellis2014minimal} it follows that $T_1 =1$. A labeling is shown in Figure \ref{fig:2xntrianglelattices1}. Thus, we will also assume $m \geq3$.  

%\begin{figure}[h!] \centering \includegraphics[scale=.5]{TriangleTube1}\caption{Scenario 1 assembly design of $2 \times n$ Triangle Lattice Tube Graph}\label{fig:2xntrianglelattices1} \end{figure}

\begin{figure}[h] \centering
\begin{tikzpicture}[transform shape, scale = .85,
  arrow inside/.style = {
    postaction={decorate},
    decoration={markings, mark=at position 0.95 with {\arrow[scale=2]{stealth}}}
  }
  ]
        %bottom side:
        \draw(0,0)--node[below]{$a$}(1,0);
        \draw(1,0)--node[below]{$\hat{a}$}(2,0);
        \draw(2,0)--node[below]{$a$}(3,0);
        \draw(3,0)--node[below]{$\hat{a}$}(4,0);
        \draw(4,0)--node[below]{$a$}(5,0);
        \draw(5,0)--node[below]{$\hat{a}$}(6,0);
        %diagonals first row:
        \draw(0,0)--node[below]{$a$}(1,1);
        \draw(1,1)--node[below]{$\hat{a}$}(2,2);
         \draw(2,0)--node[below]{$a$}(3,1);
        \draw(3,1)--node[below]{$\hat{a}$}(4,2);
         \draw(4,0)--node[below]{$a$}(5,1);
        \draw(5,1)--node[below]{$\hat{a}$}(6,2);
        %left side:
        \draw[arrow inside](0,0)--node[left]{$\hat{a}$}(0,1);
        \draw(0,1)--node[left]{$a$}(0,2); 
        %right side:
        \draw[arrow inside](6,0)--node[right]{$\hat{a}$}(6,1);
        \draw(6,1)--node[right]{$a$}(6,2);
        %interior bottom row:
        \draw(2,0)--node[right]{$\hat{a}$}(2,1);
        \draw(2,1)--node[right]{$a$}(2,2);
        \draw(4,0)--node[right]{$\hat{a}$}(4,1);
        \draw(4,1)--node[right]{$a$}(4,2);
        %top side:
        \draw(0,2)--node[above]{$a$}(1,2);
        \draw(1,2)--node[above]{$\hat{a}$}(2,2);
        \draw(2,2)--node[above]{$a$}(3,2);
        \draw(3,2)--node[above]{$\hat{a}$}(4,2);
        \draw(4,2)--node[above]{$a$}(5,2);
        \draw(5,2)--node[above]{$\hat{a}$}(6,2);
        
        \filldraw (0,0) circle [radius=3pt];
        \filldraw (0,2) circle [radius=3pt];
        \filldraw (2,0) circle [radius=3pt];
        \filldraw (2,2) circle [radius=3pt];
        \filldraw (4,0) circle [radius=3pt];
        \filldraw (4,2) circle [radius=3pt];
        \filldraw (6,0) circle [radius=3pt];
        \filldraw (6,2) circle [radius=3pt];
        \end{tikzpicture} \caption{Scenario 1 assembly design of $2 \times 4$ triangle lattice tube graph}\label{fig:2xntrianglelattices1} \end{figure}
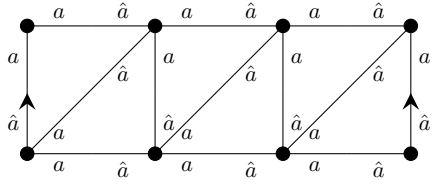

\begin{proposition}
Let $G$ be a $m \times n$ triangle lattice tube graph  with $n \geq4$,  and $m \geq 3$. Then $T_1(G)=2$. 
\end{proposition}

\begin{proof}
 By Theorem 1 in \cite{ellis2014minimal}, $av(G)\leq T_1(G) \leq ev(G) + 2ov(G)$. By inspection all vertices have degree four or six. Hence, $2\leq T_1(G) \leq 2$. The pot $P=\left\{\{a^2, \hat{a}^2\}, \{a^3, \hat{a}^3\}\right\}$ will construct $G$ by labeling parallel edges the same orientation, as in Figure~\ref{fig:largetriangletube}. %fixing the direction of one edge, and then alternating labels on consecutive arms in clockwise order about the vertices, as in Figure~\ref{fig:largetriangletube}. This is always possible since a tube may be embedded on a sphere which is an orientable surface. 
\end{proof}

\begin{figure}[h] \centering
\begin{tikzpicture}[transform shape, scale = .85,
  arrow inside/.style = {
    postaction={decorate},
    decoration={markings, mark=at position 0.95 with {\arrow[scale=2]{stealth}}}
  }
  ]
        %bottom side:
        \draw(0,0)--node[below]{$a$}(1,0);
        \draw(1,0)--node[below]{$\hat{a}$}(2,0);
        \draw(2,0)--node[below]{$a$}(3,0);
        \draw(3,0)--node[below]{$\hat{a}$}(4,0);
        \draw(4,0)--node[below]{$a$}(5,0);
        \draw(5,0)--node[below]{$\hat{a}$}(6,0);
        \draw(6,0)--node[below]{$a$}(7,0);
        \draw(7,0)--node[below]{$\hat{a}$}(8,0);
        %diagonals first row:
        \draw(0,0)--node[below]{$a$}(1,1);
        \draw(1,1)--node[below]{$\hat{a}$}(2,2);
         \draw(2,0)--node[below]{$a$}(3,1);
        \draw(3,1)--node[below]{$\hat{a}$}(4,2);
         \draw(4,0)--node[below]{$a$}(5,1);
        \draw(5,1)--node[below]{$\hat{a}$}(6,2);
         \draw(6,0)--node[below]{$a$}(7,1);
        \draw(7,1)--node[below]{$\hat{a}$}(8,2);
        %diagonals second row:
        \draw(0,2)--node[below]{$a$}(1,3);
        \draw(1,3)--node[below]{$\hat{a}$}(2,4);
         \draw(2,2)--node[below]{$a$}(3,3);
        \draw(3,3)--node[below]{$\hat{a}$}(4,4);
         \draw(4,2)--node[below]{$a$}(5,3);
        \draw(5,3)--node[below]{$\hat{a}$}(6,4);
         \draw(6,2)--node[below]{$a$}(7,3);
        \draw(7,3)--node[below]{$\hat{a}$}(8,4);
        %diagonals third row:
        \draw(0,4)--node[below]{$a$}(1,5);
        \draw(1,5)--node[below]{$\hat{a}$}(2,6);
         \draw(2,4)--node[below]{$a$}(3,5);
        \draw(3,5)--node[below]{$\hat{a}$}(4,6);
         \draw(4,4)--node[below]{$a$}(5,5);
        \draw(5,5)--node[below]{$\hat{a}$}(6,6);
         \draw(6,4)--node[below]{$a$}(7,5);
        \draw(7,5)--node[below]{$\hat{a}$}(8,6);
        %diagonals top row:
        \draw(0,6)--node[below]{$a$}(1,7);
        \draw(1,7)--node[below]{$\hat{a}$}(2,8);
         \draw(2,6)--node[below]{$a$}(3,7);
        \draw(3,7)--node[below]{$\hat{a}$}(4,8);
         \draw(4,6)--node[below]{$a$}(5,7);
        \draw(5,7)--node[below]{$\hat{a}$}(6,8);
         \draw(6,6)--node[below]{$a$}(7,7);
        \draw(7,7)--node[below]{$\hat{a}$}(8,8);
        %left side:
        \draw[arrow inside](0,0)--node[left]{$\hat{a}$}(0,1);
        \draw(0,1)--node[left]{$a$}(0,2); 
        \draw[arrow inside](0,2)--node[left]{$\hat{a}$}(0,3);
        \draw(0,3)--node[left]{$a$}(0,4);
        \draw[arrow inside](0,4)--node[left]{$\hat{a}$}(0,5);
        \draw(0,5)--node[left]{$a$}(0,6);
        \draw[arrow inside](0,6)--node[left]{$\hat{a}$}(0,7);
        \draw(0,7)--node[left]{$a$}(0,8);
        %top side:
        \draw(0,8)--node[above]{$a$}(1,8);
        \draw(1,8)--node[above]{$\hat{a}$}(2,8);
        \draw(2,8)--node[above]{$a$}(3,8);
        \draw(3,8)--node[above]{$\hat{a}$}(4,8);
        \draw(4,8)--node[above]{$a$}(5,8);
        \draw(5,8)--node[above]{$\hat{a}$}(6,8);
        \draw(6,8)--node[above]{$a$}(7,8);
        \draw(7,8)--node[above]{$\hat{a}$}(8,8);
        %right side:
        \draw[arrow inside](8,0)--node[right]{$\hat{a}$}(8,1);
        \draw(8,1)--node[right]{$a$}(8,2);
        \draw[arrow inside](8,2)--node[right]{$\hat{a}$}(8,3);
        \draw(8,3)--node[right]{$a$}(8,4);
        \draw[arrow inside](8,4)--node[right]{$\hat{a}$}(8,5);
        \draw(8,5)--node[right]{$a$}(8,6);
        \draw[arrow inside](8,6)--node[right]{$\hat{a}$}(8,7);
        \draw(8,7)--node[right]{$a$}(8,8);
        %interior bottom row:
        \draw(2,0)--node[right]{$\hat{a}$}(2,1);
        \draw(2,1)--node[right]{$a$}(2,2);
        \draw(4,0)--node[right]{$\hat{a}$}(4,1);
        \draw(4,1)--node[right]{$a$}(4,2);
        \draw(0,2)--node[below]{$a$}(1,2);
        \draw(1,2)--node[below]{$\hat{a}$}(2,2);
        \draw(2,2)--node[below]{$a$}(3,2);
        \draw(3,2)--node[below]{$\hat{a}$}(4,2);
        \draw(4,2)--node[below]{$a$}(5,2);
        \draw(5,2)--node[below]{$\hat{a}$}(6,2);
        \draw(6,0)--node[right]{$\hat{a}$}(6,1);
        \draw(6,1)--node[right]{$a$}(6,2);
        \draw(6,2)--node[below]{$a$}(7,2);
        \draw(7,2)--node[below]{$\hat{a}$}(8,2);
        %interior second row:
        \draw(2,2)--node[right]{$\hat{a}$}(2,3);
        \draw(2,3)--node[right]{$a$}(2,4);
        \draw(4,2)--node[right]{$\hat{a}$}(4,3);
        \draw(4,3)--node[right]{$a$}(4,4);
        \draw(0,4)--node[below]{$a$}(1,4);
        \draw(1,4)--node[below]{$\hat{a}$}(2,4);
        \draw(2,4)--node[below]{$a$}(3,4);
        \draw(3,4)--node[below]{$\hat{a}$}(4,4);
        \draw(4,4)--node[below]{$a$}(5,4);
        \draw(5,4)--node[below]{$\hat{a}$}(6,4);
        \draw(6,2)--node[right]{$\hat{a}$}(6,3);
        \draw(6,3)--node[right]{$a$}(6,4);
        \draw(6,4)--node[below]{$a$}(7,4);
        \draw(7,4)--node[below]{$\hat{a}$}(8,4);
        %interior third row:
        \draw(2,4)--node[right]{$\hat{a}$}(2,5);
        \draw(2,5)--node[right]{$a$}(2,6);
        \draw(4,4)--node[right]{$\hat{a}$}(4,5);
        \draw(4,5)--node[right]{$a$}(4,6);
        \draw(6,4)--node[right]{$\hat{a}$}(6,5);
        \draw(6,5)--node[right]{$a$}(6,6);
        \draw(6,6)--node[below]{$a$}(7,6);
        \draw(7,6)--node[below]{$\hat{a}$}(8,6);
        %interior top row:
        \draw(0,6)--node[below]{$a$}(1,6);
        \draw(1,6)--node[below]{$\hat{a}$}(2,6);
        \draw(2,6)--node[below]{$a$}(3,6);
        \draw(3,6)--node[below]{$\hat{a}$}(4,6);
        \draw(4,6)--node[below]{$a$}(5,6);
        \draw(5,6)--node[below]{$\hat{a}$}(6,6);
        \draw(2,6)--node[right]{$\hat{a}$}(2,7);
        \draw(2,7)--node[right]{$a$}(2,8);
        \draw(4,6)--node[right]{$\hat{a}$}(4,7);
        \draw(4,7)--node[right]{$a$}(4,8);
        \draw(6,6)--node[right]{$\hat{a}$}(6,7);
        \draw(6,7)--node[right]{$a$}(6,8);
        
        \filldraw (0,0) circle [radius=3pt];
        \filldraw (0,2) circle [radius=3pt];
        \filldraw (0,4) circle [radius=3pt];
        \filldraw (0,6) circle [radius=3pt];
        \filldraw (0,8) circle [radius=3pt];
        \filldraw (2,0) circle [radius=3pt];
        \filldraw (2,2) circle [radius=3pt];
        \filldraw (2,4) circle [radius=3pt];
        \filldraw (2,6) circle [radius=3pt];
        \filldraw (2,8) circle [radius=3pt];
        \filldraw (4,0) circle [radius=3pt];
        \filldraw (4,2) circle [radius=3pt];
        \filldraw (4,4) circle [radius=3pt];
        \filldraw (4,6) circle [radius=3pt];
        \filldraw (4,8) circle [radius=3pt];
        \filldraw (6,0) circle [radius=3pt];
        \filldraw (6,2) circle [radius=3pt];
        \filldraw (6,4) circle [radius=3pt];
        \filldraw (6,6) circle [radius=3pt];
        \filldraw (6,8) circle [radius=3pt];
        \filldraw (8,0) circle [radius=3pt];
        \filldraw (8,2) circle [radius=3pt];
        \filldraw (8,4) circle [radius=3pt];
        \filldraw (8,6) circle [radius=3pt];
        \filldraw (8,8) circle [radius=3pt];
        \end{tikzpicture} \caption{Scenario 1 assembly design of $5 \times 5$ triangle lattice tube graph}\label{fig:largetriangletube} \end{figure}
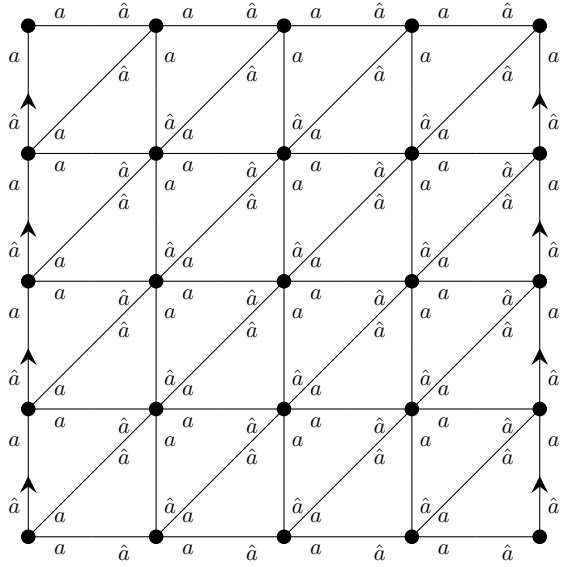 

Again, since tubes in the case that $n \leq 3$ are degenerate, we will assume that $n \geq 4$ in the following proposition.

\begin{proposition} Let $G$ be an $m \times n$ square lattice tube graph. Then $T_1(G)=3$. \end{proposition}

\begin{proof} 
%\smartqed
Note that the valency sequence of $G$ is $(3,4)$, so $2 \leq T_1(G) \leq 3$ \cite{ellis2014minimal}. Assume $P$ is a pot such that $\#P = 2$ and $G \in \mathcal{O}(P)$. Let $t_1$ denote the 3-armed tile and $t_2$ denote the 4-armed tile. The following equation must be satisfied:
$$2(n-1)z_{1,1}+(m-2)(n-2)z_{1,2}=0$$
Notice that in $G$ degree three vertices are adjacent to one another and degree four vertices are adjacent to one another, so $z_{1,1} \in \{\pm 1\}$ and $z_{1,2} \in \{0,\pm 2\}$. Note that $z_{1,2} \neq 0$, as this would imply $2(n-1)z_{1,1} = 0$, which is impossible since $n \neq 1$. Thus, $m$ and $n$ must satisfy the following equation:
\begin{equation}\label{eq:sqtubes} | 2(n-1) | = | 2(m-2)(n-2)|\end{equation}
The only integer solution pairs $(m,n)$ to Equation \ref{eq:sqtubes} are $(2,1)$ and $(4,3)$. In both cases $n<4$, so there exists no pot $P$ such that $\#P = 2$ and $G \in \mathcal{O}(P)$. The pot $\left\{\{a^2,\hat{a}\},\{a,\hat{a}^2\},\{a^2,\hat{a}^2\}\right\}$ realizes any square lattice tube $G$. An example is shown in Figure \ref{fig:squaretube}.
%\qed
\end{proof}

\begin{figure}  \centering
\begin{tikzpicture}[transform shape, scale = .85,
  arrow inside/.style = {
    postaction={decorate},
    decoration={markings, mark=at position 0.95 with {\arrow[scale=2]{stealth}}}
  }
  ]
        %bottom side:
        \draw(0,0)--node[below]{$a$}(1,0);
        \draw(1,0)--node[below]{$\hat{a}$}(2,0);
        \draw(2,0)--node[below]{$\hat{a}$}(3,0);
        \draw(3,0)--node[below]{$a$}(4,0);
        \draw(4,0)--node[below]{$a$}(5,0);
        \draw(5,0)--node[below]{$\hat{a}$}(6,0);
        \draw(6,0)--node[below]{$\hat{a}$}(7,0);
        \draw(7,0)--node[below]{$a$}(8,0);
        %left side:
        \draw[arrow inside] (0,0)--node[left]{$\hat{a}$}(0,1);
        \draw (0,1)--node[left]{$a$}(0,2); 
        \draw[arrow inside](0,2)--node[left]{$a$}(0,3);
        \draw(0,3)--node[left]{$\hat{a}$}(0,4);
        \draw[arrow inside](0,4)--node[left]{$\hat{a}$}(0,5);
        \draw(0,5)--node[left]{$a$}(0,6);
        \draw[arrow inside](0,6)--node[left]{$a$}(0,7);
        \draw(0,7)--node[left]{$\hat{a}$}(0,8);
        %top side:
        \draw(0,8)--node[above]{$a$}(1,8);
        \draw(1,8)--node[above]{$\hat{a}$}(2,8);
        \draw(2,8)--node[above]{$\hat{a}$}(3,8);
        \draw(3,8)--node[above]{$a$}(4,8);
        \draw(4,8)--node[above]{$a$}(5,8);
        \draw(5,8)--node[above]{$\hat{a}$}(6,8);
        \draw(6,8)--node[above]{$\hat{a}$}(7,8);
        \draw(7,8)--node[above]{$a$}(8,8);
        %right side:
        \draw[arrow inside](8,0)--node[right]{$\hat{a}$}(8,1);
        \draw(8,1)--node[right]{$a$}(8,2);
        \draw[arrow inside](8,2)--node[right]{$a$}(8,3);
        \draw(8,3)--node[right]{$\hat{a}$}(8,4);
        \draw[arrow inside](8,4)--node[right]{$\hat{a}$}(8,5);
        \draw(8,5)--node[right]{$a$}(8,6);
        \draw[arrow inside](8,6)--node[right]{$a$}(8,7);
        \draw(8,7)--node[right]{$\hat{a}$}(8,8);
        %interior bottom row:
        \draw(2,0)--node[left]{$a$}(2,1);
        \draw(2,1)--node[left]{$\hat{a}$}(2,2);
        \draw(4,0)--node[left]{$\hat{a}$}(4,1);
        \draw(4,1)--node[left]{$a$}(4,2);
        \draw(0,2)--node[below]{$\hat{a}$}(1,2);
        \draw(1,2)--node[above]{$a$}(2,2);
        \draw(2,2)--node[below]{$a$}(3,2);
        \draw(3,2)--node[above]{$\hat{a}$}(4,2);
        \draw(4,2)--node[below]{$\hat{a}$}(5,2);
        \draw(5,2)--node[above]{$a$}(6,2);
        \draw(6,0)--node[left]{$a$}(6,1);
        \draw(6,1)--node[left]{$\hat{a}$}(6,2);
        \draw(6,2)--node[below]{$a$}(7,2);
        \draw(7,2)--node[below]{$\hat{a}$}(8,2);
        %interior second row:
        \draw(2,2)--node[right]{$\hat{a}$}(2,3);
        \draw(2,3)--node[left]{$a$}(2,4);
        \draw(4,2)--node[right]{$a$}(4,3);
        \draw(4,3)--node[left]{$\hat{a}$}(4,4);
        \draw(0,4)--node[below]{$a$}(1,4);
        \draw(1,4)--node[above]{$\hat{a}$}(2,4);
        \draw(2,4)--node[below]{$\hat{a}$}(3,4);
        \draw(3,4)--node[above]{$a$}(4,4);
        \draw(4,4)--node[below]{$a$}(5,4);
        \draw(5,4)--node[above]{$\hat{a}$}(6,4);
        \draw(6,2)--node[right]{$\hat{a}$}(6,3);
        \draw(6,3)--node[left]{$a$}(6,4);
        \draw(6,4)--node[below]{$\hat{a}$}(7,4);
        \draw(7,4)--node[below]{$a$}(8,4);
        %interior third row:
        \draw(2,4)--node[right]{$a$}(2,5);
        \draw(2,5)--node[left]{$\hat{a}$}(2,6);
        \draw(4,4)--node[right]{$\hat{a}$}(4,5);
        \draw(4,5)--node[left]{$a$}(4,6);
        \draw(6,4)--node[right]{$a$}(6,5);
        \draw(6,5)--node[left]{$\hat{a}$}(6,6);
        \draw(6,6)--node[below]{$a$}(7,6);
        \draw(7,6)--node[below]{$\hat{a}$}(8,6);
        %interior top row:
        \draw(0,6)--node[above]{$\hat{a}$}(1,6);
        \draw(1,6)--node[above]{$a$}(2,6);
        \draw(2,6)--node[below]{$a$}(3,6);
        \draw(3,6)--node[above]{$\hat{a}$}(4,6);
        \draw(4,6)--node[below]{$\hat{a}$}(5,6);
        \draw(5,6)--node[above]{$a$}(6,6);
        \draw(2,6)--node[right]{$\hat{a}$}(2,7);
        \draw(2,7)--node[right]{$a$}(2,8);
        \draw(4,6)--node[right]{$a$}(4,7);
        \draw(4,7)--node[right]{$\hat{a}$}(4,8);
        \draw(6,6)--node[right]{$\hat{a}$}(6,7);
        \draw(6,7)--node[right]{$a$}(6,8);
        
        \filldraw (0,0) circle [radius=3pt];
        \filldraw (0,2) circle [radius=3pt];
        \filldraw (0,4) circle [radius=3pt];
        \filldraw (0,6) circle [radius=3pt];
        \filldraw (0,8) circle [radius=3pt];
        \filldraw (2,0) circle [radius=3pt];
        \filldraw (2,2) circle [radius=3pt];
        \filldraw (2,4) circle [radius=3pt];
        \filldraw (2,6) circle [radius=3pt];
        \filldraw (2,8) circle [radius=3pt];
        \filldraw (4,0) circle [radius=3pt];
        \filldraw (4,2) circle [radius=3pt];
        \filldraw (4,4) circle [radius=3pt];
        \filldraw (4,6) circle [radius=3pt];
        \filldraw (4,8) circle [radius=3pt];
        \filldraw (6,0) circle [radius=3pt];
        \filldraw (6,2) circle [radius=3pt];
        \filldraw (6,4) circle [radius=3pt];
        \filldraw (6,6) circle [radius=3pt];
        \filldraw (6,8) circle [radius=3pt];
        \filldraw (8,0) circle [radius=3pt];
        \filldraw (8,2) circle [radius=3pt];
        \filldraw (8,4) circle [radius=3pt];
        \filldraw (8,6) circle [radius=3pt];
        \filldraw (8,8) circle [radius=3pt];
        \end{tikzpicture} \caption{Scenario 1 assembly design of $5 \times 5$ square lattice tube graph }\label{fig:squaretube} \end{figure}
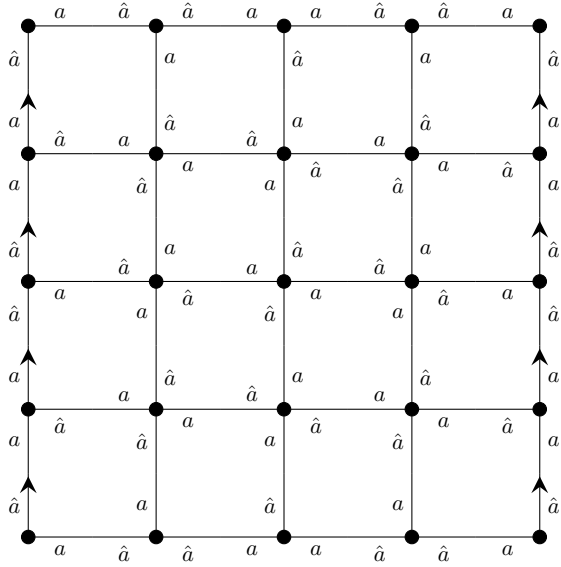 

\section{Conclusion}\label{conclusion}

We have sought here to build theoretical foundations for the field of DNA self-assembly.  Since abstractly formulated, these foundations are not limited to DNA self-assembly, but can inform any self-assembly process, at any scale, based on  building blocks with controlled cohesion sites.  We have shown that the general problem of determining the output of a pot, and of determining if a pot that realizes a given target will also realize smaller unwanted structures, are intractable.  This necessitates pragmatic solutions in the form of algorithms and closed form solutions for optimal design strategies in special situations and for specific graphs and graphs classes.    We have illustrated the utility of the construction matrix, our \emph{Maple} code, and a variety of ad hoc methods, in providing provably optimal pots in the various scenarios for the cube graph, square and triangular lattice graphs, and square and triangular lattice tube graphs. These selected examples also illustrate some of the challenges in determining minimum values for numbers of tiles and numbers of bond-edge types in Scenarios 2 and 3, particularly since we have shown that the same pot may not achieve both minimums.  
 
Our work expands known results from a small set of standard graph families to a variety of other graph types, including those common to specific applications of laboratory based DNA self-assembly. We highlighted a few notable graphs and graph classes here, and note that further results for several other families of graphs can be found in \cite{repository}.  
Much work remains to be done in seeking algorithms of broader applicability and $B$ and $T$ values for further families of graphs.  The graph theoretical implications warrant further exploration as well.  $T(G)$  and $B(G)$ are entirely  new  graph  invariants  of  independent  mathematical  and computational interest,  which naturally suggests questions such as their relation to other graph parameters.

\subsection{Acknowledgments}
Research supported in part by the program Research Experiences for Undergraduate Faculty (REUF). REUF is a program of the American Institute of Mathematics (AIM) and the Institute for Computational and Experimental Mathematics (ICERM), made possible by the support from the National Science Foundation (NSF) through DMS 1239280. Our research with students was supported in part by the following: Colonel Stephen S. and Lyla Doherty Center for Aviation and Health Research Grant, Lewis University Caterpillar Grant, The Dr. James Girard Summer Undergraduate Research Grant, Grant \# P20GM103499 (SC INBRE) from the National Institute of General Medical Sciences, National Institutes of Health.\par
Several of our students have aided in furthering our body of knowledge on this project, producing results in one or more scenarios for a variety of other niche graph families. To this end, we thank: Paul Buldak, 
Keller Dellinger, 
Hector Dondiego,   
Lauren Gernes, 
Ernesto Gonzalez, 
Chloe Griffin, 
Jackson Hansen,  
Brandon Joutras, 
Andrew Lavengood-Ryan, 
Gabriel Lopez, 
Chiara Mattamira, 
Samantha Mauro,
Sydney Martin,
Miles Mena, 
Simon Merheb, 
MeiRose Neal, 
Kayla Noon, 
Audrey Pearson, 
Heather Ray, 
Eric Redmon, 
Alvi Renzyl Cortez, 
Nick Soto, 
Adrian Siwy, 
James Sparks, 
Tyler Starkus, 
Chandler Stimpert, 
Quinn Stratton, 
Megan Vesta,
and
Jessica Williams.
\appendix
\label{appendix}

\newpage

\bibliographystyle{plain} 
\bibliography{REUFbib}

\begin{thebibliography}{10}

\bibitem{adleman1994molecular}
Leonard~M. Adleman.
\newblock Molecular computation of solutions to combinatorial problems.
\newblock {\em Science}, 266(5187):1021--1024, 1994.

\bibitem{MapleProgram}
Leyda Almod\'{o}var, Jo~Ellis-Monaghan, Amanda Harsy, Cory Johnson, and Jessica
  Sorrells.
\newblock {Spectrum of the pot with 1-2 free variables}
  \url{https://github.com/am2an7da9/spectrum-of-the-pot-with-1-2-free-variables}.

\bibitem{repository}
Leyda Almod\'{o}var, Jo~Ellis-Monaghan, Amanda Harsy, Cory Johnson, and Jessica
  Sorrells.
\newblock Optimal tile-based dna self-assembly designs for lattice graphs and
  platonic solids.
\newblock 2021.
\newblock In preparation.

\bibitem{Almodovar2019}
Leyda Almod\'{o}var, Samantha Mauro, Sydney Martin, and Heiko Todt.
\newblock Minimal tile and bond-edge types for self-assembling {DNA} graphs of
  triangular lattice graphs.
\newblock {\em Congressus Numeratium}, 232:241--263, 2019.

\bibitem{andersen2015towards}
Jakob~L. Andersen, Christoph Flamm, Martin~M. Hanczyc, and Daniel Merkle.
\newblock Towards optimal {DNA}-templated computing.
\newblock {\em International Journal of Unconventional Computing}, 11, 2015.

\bibitem{benson2015DNA}
Erik Benson, Abdulmelik Mohammed, Johan Gardell, Sergej Masich, Eugen Czeizler,
  Pekka Orponen, and Bj{\"o}rn H{\"o}gberg.
\newblock Dna rendering of polyhedral meshes at the nanoscale.
\newblock {\em Nature}, 523(7561):441--444, 2015.

\bibitem{BF2020}
Simona Bonvicini and Margherita~Maria Ferrari.
\newblock On the minimum number of bond-edge types and tile types: an approach
  by edge-colorings of graphs.
\newblock {\em Discrete Appl. Math.}, 277:1--13, 2020.

\bibitem{chen1991synthesis}
Junghuei Chen and Nadrian~C. Seeman.
\newblock Synthesis from {DNA} of a molecule with the connectivity of a cube.
\newblock {\em Nature}, 350(6319):631--633, 1991.

\bibitem{Dailey1980}
David~P. Dailey.
\newblock Uniqueness of colorability and colorability of planar {$4$}-regular
  graphs are {NP}-complete.
\newblock {\em Discrete Math.}, 30(3):289--293, 1980.

\bibitem{ellis2019tile}
Joanna Ellis-Monaghan, Nata{\v{s}}a Jonoska, and Greta Pangborn.
\newblock Tile-based {DNA} nanostructures: Mathematical design and problem
  encoding.
\newblock In {\em Algebraic and Combinatorial Computational Biology}, pages
  35--60. Elsevier, 2019.

\bibitem{ellis2013example}
Joanna Ellis-Monaghan and Greta Pangborn.
\newblock An example of practical organization for undergraduate research
  experiences.
\newblock {\em PRIMUS}, 23(9):805--814, 2013.

\bibitem{ellis2014minimal}
Joanna Ellis-Monaghan, Greta Pangborn, Laura Beaudin, David Miller, Nick Bruno,
  and Akie Hashimoto.
\newblock Minimal tile and bond-edge types for self-assembling {DNA} graphs.
\newblock In {\em Discrete and Topological Models in Molecular Biology}, pages
  241--270. Springer, 2014.

\bibitem{CW2017}
Constantine~G. Evans and Erik Winfree.
\newblock Physical principles for dna tile self-assembly.
\newblock {\em Chem. Soc. Rev.}, 46:3808--3829, 2017.

\bibitem{ferrari2018}
Margherita Ferrari, Anna Cook, Alana Houlihan, Rebecca Roulaeu, Nadrian Seeman,
  Greta Pangborn, and Joanna Ellis-Monaghan.
\newblock Design formalism for {DNA} self-assembly of polyhedral skeletons
  using rigid tiles.
\newblock {\em The Journal of Mathematical Chemistry}, 56(5):1365--1392, 2018.

\bibitem{funke20162}
Jonas~J. Funke, Philip Ketterer, Corinna Lieleg, Philipp Korber, and Hendrik
  Dietz.
\newblock Exploring nucleosome unwrapping using {DNA} origami.
\newblock {\em Nano Letters}, 16(12):7891--7898, 2016.

\bibitem{funke2016}
Jonas~J. Funke, Philip Ketterer, Corinna Lieleg, Sarah Schunter, Philipp
  Korber, and Hendrik Dietz.
\newblock Uncovering the forces between nucleosomes using {DNA} origami.
\newblock {\em Science Advances}, 2(11), 2016.

\bibitem{41}
Andre~K. Geim and Konstantin~S. Novoselov.
\newblock The rise of graphene.
\newblock In {\em Nanoscience and Technology: A Collection of Reviews from
  Nature Journals}, pages 11--19. World Scientific, 2010.

\bibitem{gerling2015}
Thomas Gerling, Klaus~F. Wagenbauer, Andrea~M. Neuner, and Hendrik Dietz.
\newblock Dynamic {DNA} devices and assemblies formed by shape-complementary,
  non{\textendash}base pairing {3D} components.
\newblock {\em Science}, 347(6229):1446--1452, 2015.

\bibitem{goodnow2017dna}
Robert~A. Goodnow, Christoph~E. Dumelin, and Anthony~D. Keefe.
\newblock {DNA}-encoded chemistry: enabling the deeper sampling of chemical
  space.
\newblock {\em Nature Reviews Drug Discovery}, 16(2):131--147, 2017.

\bibitem{43}
Hongzhou Gu, Jie Chao, Shou-Jun Xiao, and Nadrian~C. Seeman.
\newblock A proximity-based programmable {DNA} nanoscale assembly line.
\newblock {\em Nature}, 465(7295):202, 2010.

\bibitem{hansen2018DNA}
Bjarke~N. Hansen, Kim~S. Larsen, Daniel Merkle, and Alexei Mihalchuk.
\newblock {DNA}-templated synthesis optimization.
\newblock {\em Natural Computing}, 17(4):693--707, 2018.

\bibitem{he2008hierarchical}
Yu~He, Tao Ye, Min Su, Chuan Zhang, Alexander~E Ribbe, Wen Jiang, and Chengde
  Mao.
\newblock Hierarchical self-assembly of {DNA} into symmetric supramolecular
  polyhedra.
\newblock {\em Nature}, 452(7184):198, 2008.

\bibitem{49}
Ryosuke Iinuma, Yonggang Ke, Ralf Jungmann, Thomas Schlichthaerle, Johannes~B
  Woehrstein, and Peng Yin.
\newblock Polyhedra self-assembled from {DNA} tripods and characterized with
  {3D} {DNA}-paint.
\newblock {\em science}, page 1250944, 2014.

\bibitem{jonoska2006spectrum}
Nata{\v{s}}a Jonoska, Gregory~L. McColm, and Ana Staninska.
\newblock Spectrum of a pot for {DNA} complexes.
\newblock In {\em International Workshop on {DNA}-Based Computers}, pages
  83--94. Springer, 2006.

\bibitem{55}
Nata{\v{s}}a Jonoska, Phiset Sa-Ardyen, and Nadrian~C. Seeman.
\newblock Computation by self-assembly of {DNA} graphs.
\newblock {\em Genetic Programming and Evolvable Machines}, 4(2):123--137,
  2003.

\bibitem{59}
Neville~R. Kallenbach, Rong-Ine Ma, and Nadrian~C. Seeman.
\newblock An immobile nucleic acid junction constructed from oligonucleotides.
\newblock {\em Nature}, 305(5937):829, 1983.

\bibitem{61}
Hyunho Kim, Sungwoo Yang, Sameer~R. Rao, Shankar Narayanan, Eugene~A. Kapustin,
  Hiroyasu Furukawa, Ari~S. Umans, Omar~M. Yaghi, and Evelyn~N. Wang.
\newblock Water harvesting from air with metal-organic frameworks powered by
  natural sunlight.
\newblock {\em Science}, 356(6336):430--434, 2017.

\bibitem{labean2007constructing}
Thom~H. LaBean and Hanying Li.
\newblock Constructing novel materials with {DNA}.
\newblock {\em Nano Today}, 2(2):26--35, 2007.

\bibitem{le2016}
Jenny~V. Le, Yi~Luo, Michael~A. Darcy, Christopher~R. Lucas, Michelle~F.
  Goodwin, Michael~G. Poirier, and Carlos~E. Castro.
\newblock ``{P}robing {N}ucleosome {S}tability with a {DNA} {O}rigami
  {N}anocaliper''.
\newblock {\em ACS Nano}, 10(7):7073--7084, 2016.

\bibitem{64}
Di~Liu, Gang Chen, Usman Akhter, Timothy~M Cronin, and Yossi Weizmann.
\newblock Creating complex molecular topologies by configuring {DNA} four-way
  junctions.
\newblock {\em Nature chemistry}, 8(10):907, 2016.

\bibitem{65}
Di~Liu, Yaming Shao, Gang Chen, Yuk-Ching Tse-Dinh, Joseph~A. Piccirilli, and
  Yossi Weizmann.
\newblock Synthesizing topological structures containing rna.
\newblock {\em Nature communications}, 8:14936, 2017.

\bibitem{Liu2019}
Xiaoguo Liu, Yan Zhao, Pi~Liu, Lihua Wang, Jianping Lin, and Chunhai Fan.
\newblock Biomimetic {DNA} nanotubes: Nanoscale channel design and
  applications.
\newblock {\em Angewandte Chemie International Edition}, 58(27):8996--9011,
  2019.

\bibitem{67}
Kyle Lund, Anthony~J. Manzo, Nadine Dabby, Nicole Michelotti, Alexander
  Johnson-Buck, Jeanette Nangreave, Steven Taylor, Renjun Pei, Milan~N.
  Stojanovic, Nils~G. Walter, et~al.
\newblock Molecular robots guided by prescriptive landscapes.
\newblock {\em Nature}, 465(7295):206, 2010.

\bibitem{Mattamira}
Chiara Mattamira.
\newblock {DNA} self-assembly design for gear graphs.
\newblock {\em Rose-Hulman Undergraduate Mathematics Journal}, 21, 2020.
\newblock \url{https://scholar.rose-hulman.edu/rhumj/vol21/iss1/11}, Last
  accessed on \today.

\bibitem{69}
Tosan Omabegho, Ruojie Sha, and Nadrian~C. Seeman.
\newblock A bipedal {DNA} brownian motor with coordinated legs.
\newblock {\em Science}, 324(5923):67--71, 2009.

\bibitem{72}
John~A. Pelesko.
\newblock {\em Self assembly: the science of things that put themselves
  together}.
\newblock Chapman and Hall/CRC, 2007.

\bibitem{rothemund2006folding}
Paul~W.K. Rothemund.
\newblock Folding {DNA} to create nanoscale shapes and patterns.
\newblock {\em Nature}, 440(7082):297--302, 2006.

\bibitem{rothemund2004design}
Paul~W.K. Rothemund, Axel Ekani-Nkodo, Nick Papadakis, Ashish Kumar,
  Deborah~Kuchnir Fygenson, and Erik Winfree.
\newblock Design and characterization of programmable {DNA} nanotubes.
\newblock {\em Journal of the American Chemical Society}, 126(50):16344--16352,
  2004.

\bibitem{75}
Phiset Sa-Ardyen, Nata{\v{s}}a Jonoska, and Nadrian~C. Seeman.
\newblock Self-assembling {DNA} graphs.
\newblock In {\em International Workshop on {DNA}-Based Computers}, pages 1--9.
  Springer, 2002.

\bibitem{Seeman82}
Nadrian~C. Seeman.
\newblock Nucleic acid junctions and lattices.
\newblock {\em Journal of Theoretical Biology}, 99:237--247, 1982.

\bibitem{seeman2007overview}
Nadrian~C. Seeman.
\newblock An overview of structural {DNA} nanotechnology.
\newblock {\em Molecular biotechnology}, 37(3):246, 2007.

\bibitem{78}
Nadrian~C. Seeman.
\newblock {\em Structural {DNA} nanotechnology}.
\newblock Cambridge University Press, 2016.

\bibitem{SK94}
Nadrian~C. Seeman and Neville~R. Kallenbach.
\newblock {DNA} branched junctions.
\newblock {\em Annual Review of Biophysics and Biomolecular Structure},
  23(1):53--86, 1994.
\newblock PMID: 7919792.

\bibitem{steph2020}
Nicholas Stephanopoulos.
\newblock Hybrid nanostructures from the self-assembly of proteins and {DNA}.
\newblock {\em Chem}, 6(2):364 -- 405, 2020.

\bibitem{84}
Wei Sun, Etienne Boulais, Yera Hakobyan, Wei~Li Wang, Amy Guan, Mark Bathe, and
  Peng Yin.
\newblock Casting inorganic structures with {DNA} molds.
\newblock {\em Science}, page 1258361, 2014.

\bibitem{88}
Hui Wang, Russell~J. Di~Gate, and Nadrian~C. Seeman.
\newblock An {RNA} topoisomerase.
\newblock {\em Proceedings of the National Academy of Sciences},
  93(18):9477--9482, 1996.

\bibitem{89}
Yinli Wang, John~E. Mueller, B{\"o}rries Kemper, and Nadrian~C. Seeman.
\newblock Assembly and characterization of five-arm and six-arm {DNA} branched
  junctions.
\newblock {\em Biochemistry}, 30(23):5667--5674, 1991.

\bibitem{wickham2012DNA}
Shelley~F.J. Wickham, Jonathan Bath, Yousuke Katsuda, Masayuki Endo, Kumi
  Hidaka, Hiroshi Sugiyama, and Andrew~J. Turberfield.
\newblock A {DNA}-based molecular motor that can navigate a network of tracks.
\newblock {\em Nature nanotechnology}, 7(3):169--173, 2012.

\bibitem{wilner2011self}
Ofer~I. Wilner, Ron Orbach, Anja Henning, Carsten Teller, Omer Yehezkeli,
  Michael Mertig, Daniel Harries, and Itamar Willner.
\newblock Self-assembly of {DNA} nanotubes with controllable diameters.
\newblock {\em Nature communications}, 2:540, 2011.

\bibitem{91}
Erik Winfree.
\newblock {\em Algorithmic self-assembly of {DNA}}.
\newblock PhD thesis, California Institute of Technology, 1998.

\bibitem{92}
Erik Winfree, Furong Liu, Lisa~A. Wenzler, and Nadrian~C. Seeman.
\newblock Design and self-assembly of two-dimensional {DNA} crystals.
\newblock {\em Nature}, 394(6693):539, 1998.

\bibitem{94}
Gang Wu, Natasha Jonoska, and Nadrian~C. Seeman.
\newblock Construction of a {DNA} nano-object directly demonstrates
  computation.
\newblock {\em Biosystems}, 98(2):80--84, 2009.

\bibitem{yan2003DNA}
Hao Yan, Sung~Ha Park, Gleb Finkelstein, John~H. Reif, and Thomas~H. LaBean.
\newblock {DNA}-templated self-assembly of protein arrays and highly conductive
  nanowires.
\newblock {\em science}, 301(5641):1882--1884, 2003.

\bibitem{70}
Peng Yin, Hao Yan, Xiaoju~G Daniell, Andrew~J Turberfield, and John~H Reif.
\newblock A unidirectional {DNA} walker that moves autonomously along a track.
\newblock {\em Angewandte Chemie}, 116(37):5014--5019, 2004.

\bibitem{zhang2012}
Chuan Zhang, Cheng Tian, Fei Guo, Zheng Liu, Wen Jiang, and Chengde Mao.
\newblock {DNA}-directed three-dimensional protein organization.
\newblock {\em Angewandte Chemie Inte{rna}tional Edition}, 51(14):3382--3385,
  2012.

\bibitem{zhang1994construction}
Yuwen Zhang and Nadrian~C. Seeman.
\newblock Construction of a {DNA}-truncated octahedron.
\newblock {\em Journal of the American Chemical Society}, 116(5):1661--1669,
  1994.

\bibitem{99}
Jianping Zheng, Jens~J. Birktoft, Yi~Chen, Tong Wang, Ruojie Sha, Pamela~E.
  Constantinou, Stephan~L. Ginell, Chengde Mao, and Nadrian~C. Seeman.
\newblock From molecular to macroscopic via the rational design of a
  self-assembled {{3D} {DNA}} crystal.
\newblock {\em Nature}, 461(7260):74, 2009.

\end{thebibliography}

\end{document}